\documentclass{article}
\usepackage[utf8]{inputenc}
\usepackage{amsmath}
\usepackage{amssymb}
\usepackage{amsthm}
\usepackage{graphicx}
\usepackage{caption}
\usepackage{subcaption}
\usepackage{color}
\usepackage{comment}
\usepackage{authblk}
\usepackage{hyperref}

\makeindex
\numberwithin{equation}{section}

\newtheorem{Theorem}{Theorem}[section]
\newtheorem{Corollary}[Theorem]{Corollary}

\newtheorem{Lemma}[Theorem]{Lemma}

\theoremstyle{definition}

\usepackage{geometry}  % See geometry.pdf to learn the many layout options.

\newcommand{\spacetime}{\mathbb{R}^{1,3}}
\newcommand{\term}[1]{{\bf\em #1}}

\newcommand{\deSitter}{\textrm{d}\mathbb{S}^3}

\newcommand*{\email}[1]{%
    \normalsize\href{mailto:#1}{#1}\par
    }

\title{A proof of the Koebe-Andre'ev-Thurston theorem via flow from tangency packings}

\author{John C.~Bowers}
\affil{Department of Computer Science\\ James Madison University \\ Harrisonburg, VA 22807, USA\\ \email{bowersjc@jmu.edu}} 

\begin{document}

\maketitle

\begin{abstract}\noindent
Recently, Connelly and Gortler gave a novel proof of the circle packing theorem for tangency packings by introducing a hybrid combinatorial-geometric operation, flip-and-flow, that allows two tangency packings whose contact graphs differ by a combinatorial edge flip to be continuously deformed from one to the other while maintaining tangencies across all of their common edges. Starting from a canonical tangency circle packing with the desired number of circles a finite sequence of flip-and-flow operations may be applied to obtain a circle packing for any desired (proper) contact graph with the same number of circles.

In this paper, we extend the Connelly-Gortler method to allow circles to overlap by angles up to $\pi/2$. As a result, we obtain a new proof of the general Koebe-Andre'ev-Thurston theorem for disk packings on $\mathbb{S}^2$ with overlaps and a numerical algorithm for computing them. Our development makes use of the correspondence between circles and disks on $\mathbb{S}^2$ and hyperplanes and half-spaces in the 4-dimensional Minkowski spacetime $\spacetime$, which we illuminate in a preliminary section. Using this view we generalize a notion of convexity of circle polyhedra that has recently been used to prove the global rigidity of certain circle packings. Finally, we use this view to show that all convex circle polyhedra are infinitesimally rigid, generalizing a recent related result. 
\end{abstract}

\section{Introduction}

The celebrated Koebe-Andre'ev-Thurston theorem (KAT) states that, given a triangulation $T$ of a topological sphere together with an assignment of desired overlap angles in $[0, \pi/2]$ for each edge of the triangulation (subject to a few mild conditions), there exists a pattern of disks on the unit sphere $\mathbb{S}^2$, one disk for each vertex, such that the angle of overlap between two circles that correspond to an edge matches the desired overlap angle for that edge. Furthermore, connecting centers of neighboring disks along geodesic shortest paths produces a geodesic triangulation of $\mathbb{S}^2$ that is isomorphic with $T$, and, moreover, the disk pattern obtained by the theorem is unique up to Möbius transformations of $\mathbb{S}^2$ to itself. When all overlap angles are 0 then neighboring disks are tangent and the restricted version of this theorem is called the circle packing theorem. 

Numerous proofs of the circle packing theorem have appeared in the literature, for example \cite{pK36,Thurston:1980,MR90,BSt90,CdV91a,pB93,kS05}. Fewer proofs of the more general KAT theorem have appeared. Andre'ev had a proof \cite{Andreev:1970a,Andreev:1970b} containing errors that was later fixed by Roeder, Hubbard, and Dunbar \cite{Roeder2007}. Bowers and Stephenson proved the result as a special case of their work on branched packings \cite{BS96}. There are also several generalizations, for instance packings with deep overlap angles of more than $\pi/2$ due to Rivin \cite{Rivin:1994kg}. A nice survey on the history of the theorem and a host of related results may be found in \cite{B19}. 

The problem of circle packing was further generalized by Bowers and Stephenson \cite{BS04} to include circles that do not overlap by replacing desired overlap angles with desired {\em inversive distances}. When two circles overlap, their inversive distance is the cosine of their angle of overlap, but is also defined when they do not. In this full generality much less is known. The existence of circle patterns realizing a set of desired inversive distances corresponding to a triangulated polyhedron has not been characterized and remains an important open problem. It is known that not all possible assignments of inversive distances to a triangulation are realizable. Furthermore, unlike in the tangency and overlapping cases, general inversive distance packings are not necessarily unique (as was shown in \cite{Ma:2012hl} and further explored in \cite{Bowers2017}), though uniqueness may be obtained with an appropriate notion of convexity for a circle packing, such as that in \cite{BBP18}, which we generalize in this paper. The development of new techniques seems necessary to get a handle on circle packings in this more general setting.

Recently Connelly and Gortler produced a novel proof of the circle packing theorem that obtains a packing via a continuous motion of circles \cite{CG19}. Their method starts from a canonical packing of circles and, through a finite sequence of moves called flip-and-flow operations, smoothly varies the starting circle pattern into a tangency pattern with the same number of circles with any desired combinatorics homeomorphic to a triangulation of a sphere. Their proof is inherently algorithmic and can be numerically approximated. It is notable that their method computes both the centers and radii of the circles simultaneously. Prior algorithms for circle packing, such as \cite{CS03,OSC17}, have either computed all radii and then produced circle centers in a separate layout step once the radii where computed or used multiple passes that alternate between an adjustment of all radii followed by an adjustment of the circle centers. Connelly and Gortler's method is the only algorithm known to this author that computes both centers and radii in an integrated and simultaneous way. Following their approach we obtain a similar algorithm for computing shallow overlap circle packings that also computes circle centers and radii in an integrated way.

\paragraph{Contributions of this paper.}
In this paper we extend Connelly and Gortler's method to obtain the full generality of the KAT theorem. Our starting point is the ending point of their algorithm--a tangency packing that has the right combinatorics, but does not yet realize desired overlap angles between $0$ and $\pi/2$. Our method then uses the same flow operation as theirs without the combinatorial flip mechanism to adjust, one by one, the overlaps across each edge of the triangulation. Considerable work is required to show that the flow exists and can be carried through to any desired overlap subject to the mild conditions of KAT.

In order to generalize Connelly and Gortler's result we moved the entire process from the plane onto the sphere $\mathbb{S}^2$. Packings on the plane and on the sphere are related by stereographic projection. Our proofs are inherently spherical and the algorithm we use to adjust the inversive distances is carried out {\em in situ} on the sphere. We make significant use of the connection between triangulated polyhedra in the de Sitter space and disk packings on $\mathbb{S}^2$. This allows us to prove the infinitesimal rigidity of all strictly convex circle polyhedra. This is a partial generalization of the result from \cite{BBP19} that all c-polyhedra are rigid. This connection between circle polyhedra and de Sitter polyhedra exists as a sort of mathematical folklore--it is known to several researchers in the field and occasionally plays a role in proofs (e.g. \cite{Ma:2012hl}), but we are not aware of any work that carefully explains the connection in any level of detail. As such we hope that our preliminaries section may be found a useful introduction to this way of viewing circle packings. We remark that this view is also useful computationally, since it provides a homogeneous representation of disks on $\mathbb{S}^2$, which eliminates many special cases from algorithms dealing with circle patterns and linearalizes many properties of circle polyhedra. 

An example of producing an overlap packing using the techniques of this paper is shown in figure~\ref{fig:algorithmrun}. The code that produced this figure is not currently in a form fit to release publicly, but the author is happy to make it available to any interested readers upon request. 

\paragraph{Organization of the paper.}
We first give a short introduction to the analytic geometry of disks on $\mathbb{S}^2$ which identifies disks with spacelike rays through the origin of the Minkowski spacetime $\spacetime$ (\S\ref{sec:preliminaries}). We use this view to develop the geometry of circles/disks on the sphere. 

We then define our notion of a circle packing, which we call circle polyhedra, and develop properties of circle polyhedra, such as convexity and those that are geodesic (\S\ref{sec:circlepolyhedra}). Our notion of convexity is in some sense new to circle packing and generalizes the notion of convexity developed in \cite{BBP18}. It is defined there only for circle polyhedra with hyperbolic faces. Here we drop the restriction to hyperbolic faces and define convexity for all circle polyhedra. 

We then analyze the properties of motions of disk patterns that maintain certain constraints (\S\ref{sec:motions}). Our goal in this section is to show that, subject to some mild constraints that arise naturally from the conditions in the KAT theorem, motions of disk patterns that maintain shallow overlaps behave well--they maintain a geodesic triangulation of $\mathbb{S}^2$, remain convex, and no disk shrinks to a point or grows to encompass the entire sphere without violating a constraint. It is interesting that our analysis here is for very general motions of circle polyhedra in which no edge's inversive distance is fixed, only the range of inversive distances is constrained along each edge. Our use of these motions to prove KAT is a significantly restricted class of this more general investigation.

Next, we show that all strictly convex circle polyhedra are rigid (\S\ref{sec:infrigidity}) and develop the necessary rigidity theory of circle polyhedra following the same general outline as Connelly and Gortler. Though the development here is fairly standard for rigidity theory, we believe this result is of independent interest to the circle packing community in that it generalizes to all strictly convex circle polyhedra the recent result that almost all circle polyhedra are infinitesimally rigid \cite{BBP19}. This section includes the development of the packing manifold, a 1-dimensional Riemannian manifold of disk configurations obtained by dropping one edge constraint from a convex circle polyhedra. The packing manifold was developed in \cite{CGT19,CG19} for circle polyhedra where all edges are tangent save for the dropped edge. We extend this notion to all geodesic convex circle polyhedra with shallow overlaps. 

Finally, we show how to put the pieces together in order to extend Connelly and Gortler's method (\S\ref{sec:proofofkat}) and prove the full Koebe-Andre'ev-Thurston method. The first part of our proof follows the general outline of Connelly and Gortler, with the details provided by the preceding sections. This proves a restricted version of the KAT theorem for circle polyhedra we call strictly shallow. The full generality of the theorem is then obtained by a limiting argument. 

We end with an intriguing example obtained by computational experimentation that shows this idea extends beyond the KAT theorem. 

\begin{figure}[hbt!]
  \centering
  \begin{minipage}[b]{0.4\textwidth}
  	\includegraphics[width=\textwidth]{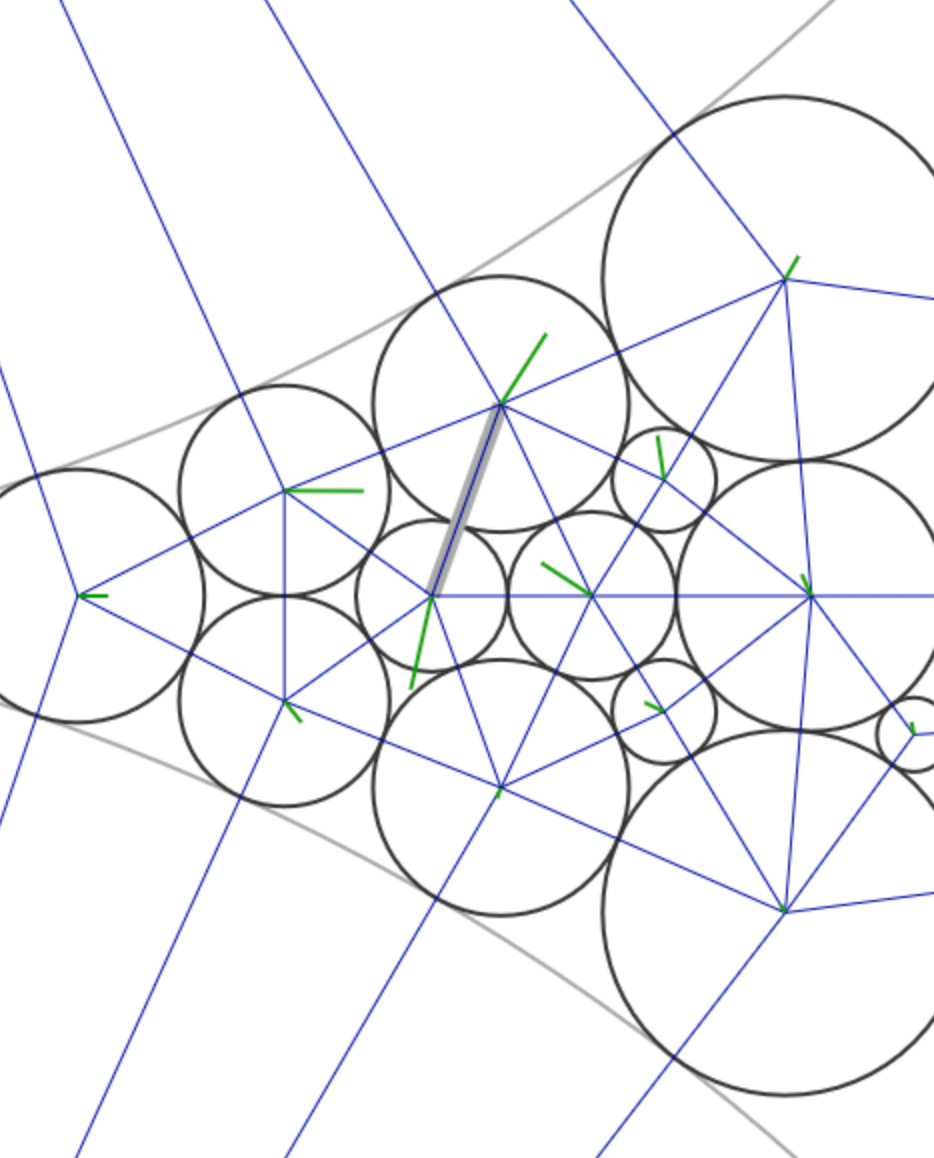}
	\subcaption{}
  \end{minipage}
  \begin{minipage}[b]{0.4\textwidth}
	  \includegraphics[width=\textwidth]{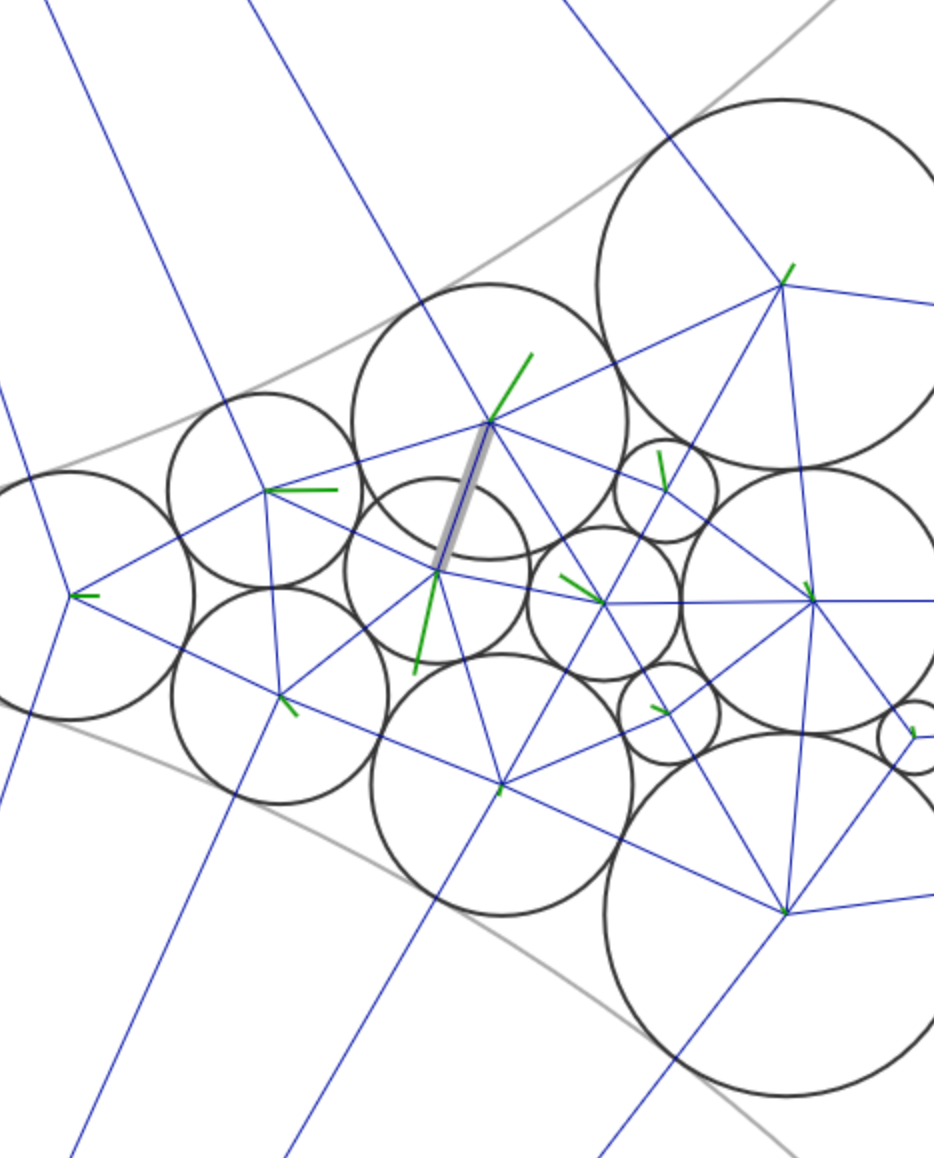}
	  \subcaption{}
  \end{minipage} \\
  \begin{minipage}[b]{0.4\textwidth}
  	\includegraphics[width=\textwidth]{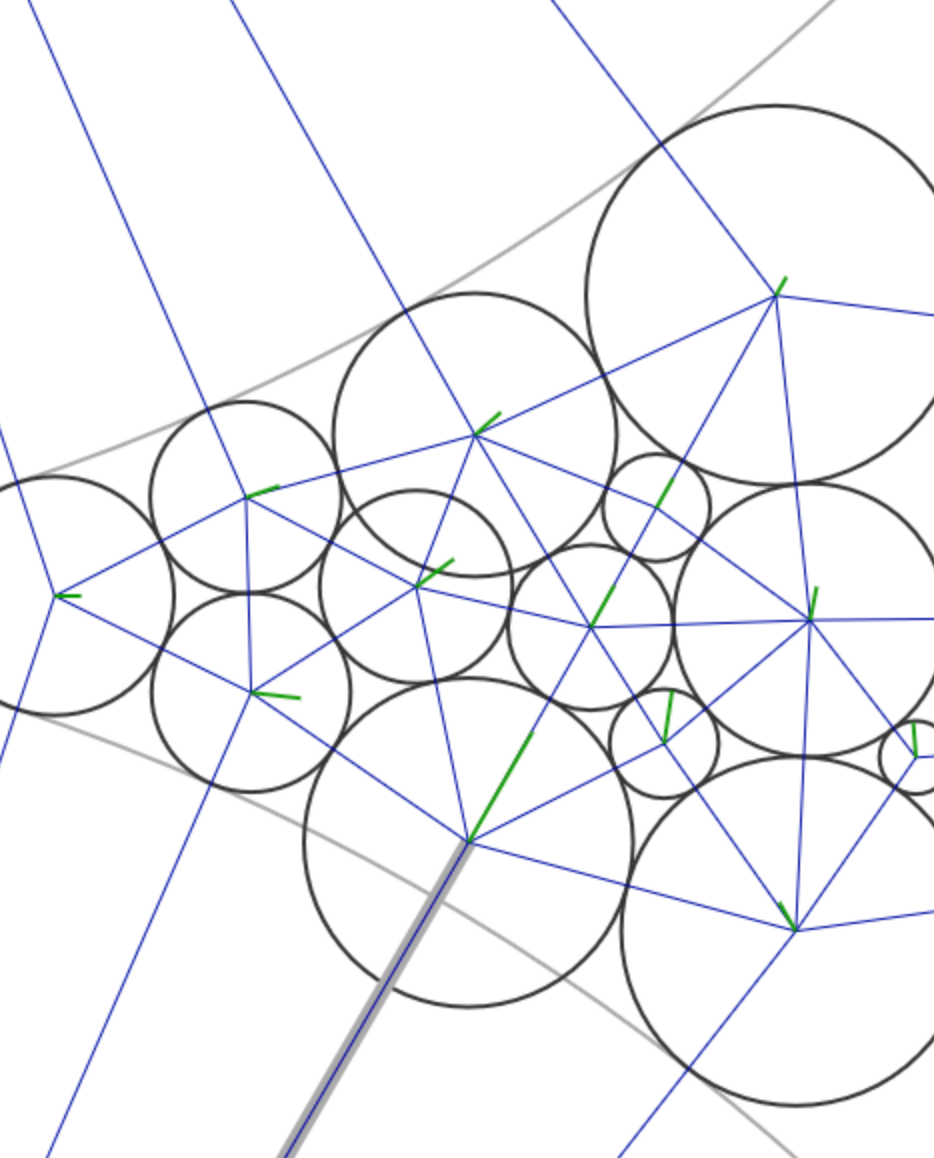}
	\subcaption{}
  \end{minipage}
  \begin{minipage}[b]{0.4\textwidth}
	  \includegraphics[width=\textwidth]{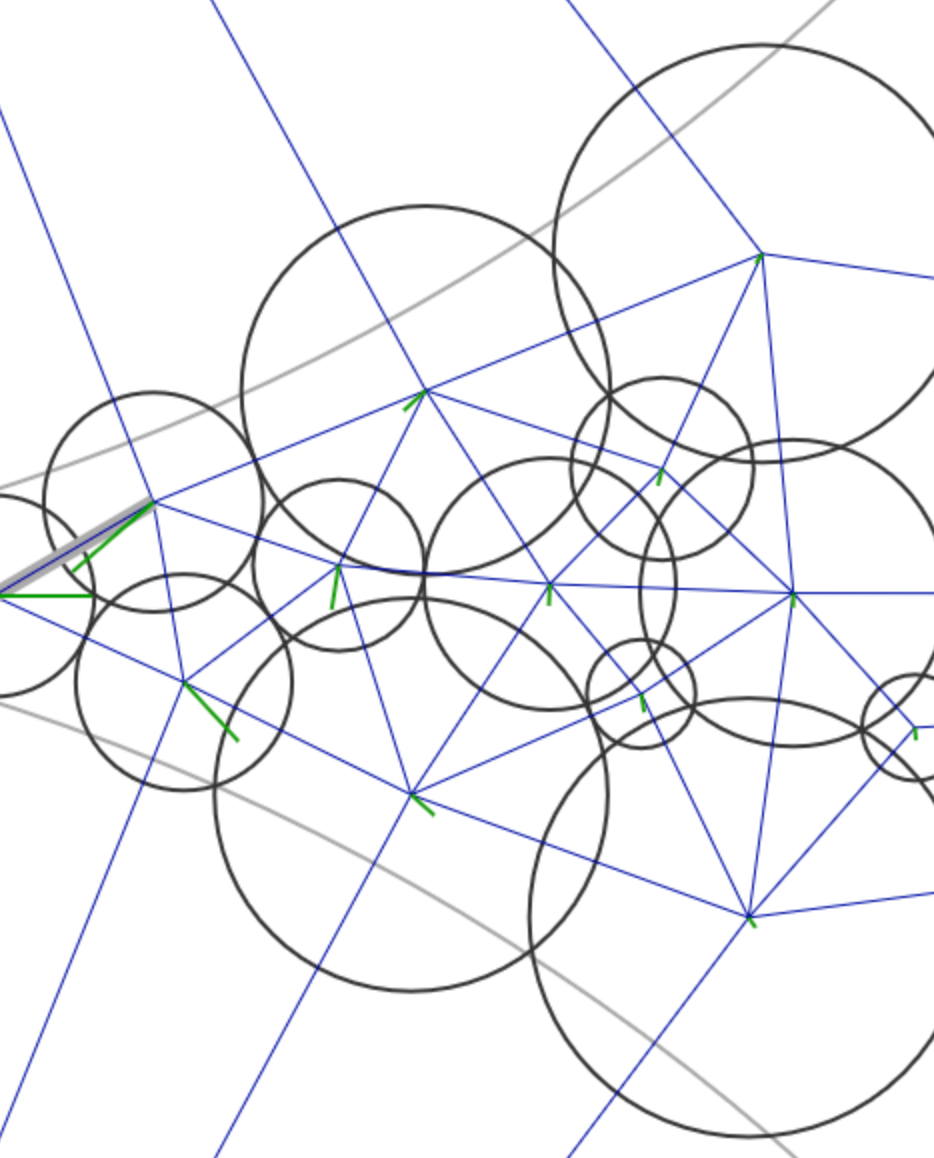}
	  \subcaption{}
  \end{minipage} \\
  
  \caption{Starting from a tangency packing (a), our proof fixes all inversive distances except one (represented by the gray edge) which is corrected via flow. (b) and (c) show two successive applications of flow along two different edges. (d) shows a final packing. The green lines are the tangent vectors to the motion of each circle's center in the flow determined by the gray free-edge.}
  \label{fig:algorithmrun}
\end{figure}

\section{Preliminaries}\label{sec:preliminaries}

\subsection{Analytic geometry of disks on $\mathbb{S}^2$}

Our main objects of interest are disks on $\mathbb{S}^2$. In order to parametrize disks on $\mathbb{S}^2$ we begin with the Minkowski spacetime, $\spacetime$, which is $\mathbb{R}^4$ endowed with the pseudo-Euclidean inner product of signature $(1, 3)$. This inner product, known as the \term{Lorentz inner product}, is given by 
\begin{equation}
	\langle (t_1, x_1, y_1, z_1), (t_2, x_2, y_2, z_2)\rangle_{1,3} = t_1 t_2 - x_1 x_2 - y_1 y_2 - z_1 z_2.
\end{equation}
It is standard to treat the $t=1$ subspace of $\spacetime$ as a model for three dimensional Euclidean space, but such a choice has certain disadvantages in our setting. Instead, we model $\mathbb{E}^3$ as the set of rays from the origin having a positive $t$-coordinate (or alternatively as the upper hemisphere of the unit sphere in $\mathbb{R}^4$). We parametrize this space by specifying a point on the ray $(t, x, y, z)$ and thus the Euclidean point corresponds to the equivalence class of points $(\lambda t, \lambda x, \lambda y, \lambda z)$ for $\lambda > 0$. To ``view'' a set of Euclidean points we compute the intersection of the ray with the $t=1$ subspace, so that the usual $\mathbb{E}^3$ coordinates for the point represented by $(t, x, y, z)$ are $(x/t, y/t, z/t)$. The advantage of viewing $\mathbb{E}^3$ as a set of rays instead of the $t=1$ subspace is that compactifying $\mathbb{E}^3$ with the sphere at infinity is natural in this view. Every ray through a point $(0, x, y, z)$ represents the point at infinity in $\mathbb{E}^3$ in the direction $(x, y, z)$. This \term{extended Euclidean space}, $\widehat{\mathbb{E}}^3$ is modeled by the rays in $\spacetime$ with $t\geq 0$, or equivalently with the upper $t$-hemisphere of the unit Euclidean 3-sphere in $\spacetime$. We could also include those rays with negative $t$-coordinate, which leads to a double covering of $\mathbb{E}^3$, in which every point of $\mathbb{E}^3$ also has an orientation $+$ or $-$, connected by the sphere at infinity. Such a construction leads to a model of the oriented projective space $T^3$, an oriented version of $\mathbb{RP}^3$, but we will not explicitly use this in this paper. See \cite{S87} for a development of oriented projective geometry. 

Let $\mathbf{x}\in\spacetime$. The \term{Minkowski norm} of $\mathbf{x}$ is the complex valued function $|\mathbf{x}|_{1,3} = \sqrt{\langle\mathbf{x}, \mathbf{x}\rangle_{1,3}}$. The points for which $|\mathbf{x}|_{1,3} = 0$ is called the \term{light cone}. The set of points for which $|\mathbf{x}|_{1,3}^2 = -1$ forms a hyperboloid of revolution around the $a$-axis in $\spacetime$ called the \term{de Sitter space}, denoted $d\mathbb{S}^3$, and the set of points for which $|\mathbf{x}|_{1,3}^2 =1$ forms a hyperboloid of two sheets which is a model for the hyperbolic space $\mathbb{H}^3$. The points $\mathbf{x}\in\spacetime$ are classified as \term{lightlike} if $|\mathbf{x}|_{1,3} = 0$, \term{spacelike} if $|\mathbf{x}|_{1,3}^2 < 0$, and \term{timelike} if $|\mathbf{x}|_{1,3}^2 < 0$. Similarly, rays, vectors, and lines through the origin of $\spacetime$ are classified as lightlike, timelike, or spacelike depending on whether the points along the ray, the endpoint of the vector, or the points of the line are lightlike, timelike, or spacelike.

\paragraph{The unit sphere $\mathbb{S}^2$.}
Consider the rays from the origin of $\spacetime$ with positive $t$-coordinate. Those that are lightlike correspond to the points of the unit sphere $\mathbb{S}^2$ in $\mathbb{E}^3$. Those that are timelike correspond to points on the interior of the sphere (and thus are a model of the hyperbolic space $\mathbb{H}^3$). Those that are spacelike correspond to points on the exterior of $\mathbb{S}^2$. 

\paragraph{Lorentz normals and hyperplanes.}
Let ${\bf N} = (a, b, c, d)$ and ${\bf x} = (t, x, y, z)$ be two vectors in $\mathbb{R}^{1,3}$. The two vectors are \term{Lorentz orthogonal}, or simply orthogonal, if and only if $\langle {\bf N}, {\bf x}\rangle_{1,3} = 0$. The set $\Pi_{\bf N} = \{{\bf x}\in\mathbb{R}^{1,3} : \langle {\bf N}, {\bf x} \rangle_{1, 3} = 0\}$ is a hyperplane through the origin of $\spacetime$ with \term{Lorentz normal}, or simply normal, ${\bf N}$. We parametrize the hyperplanes through the origin of $\mathbb{R}^{1,3}$ by their Lorentz normals. We will use $(a, b, c, d)$ to describe coefficients of the normal of a hyperplane through the origin of $\spacetime$ and $(t, x, y, z)$ to denote a point.  A hyperplane through the origin of $\spacetime$ is lightlike if it is tangent to the lightcone, timelike if it intersects the lightcone non-trivially, and spacelike if it meets the lightcone only at the origin. A lightlike hyperplane has a lightlike normal, a timelike hyperplane has a spacelike normal, and a spacelike hyperplane has a timelike normal.

\paragraph{Circles and disks on $\mathbb{S}^2$}
A circle on $\mathbb{S}^2$ is given by the intersection of a hyperplane through the origin of $\mathbb{R}^{1,3}$ with $\mathbb{S}^2$. Thus, we parametrize the set of circles on $\mathbb{S}^2$ by $(a, b, c, d)$ where the circle is the set of points $(t, x, y, z) \in \mathbb{S}^2$ satisfying
\begin{equation}
	a t - b x - c y - d z = 0. 
\end{equation}
The two disks bounded by the circle $(a, b, c, d)$ are the positively oriented disk $a t - b x - c y - d z < 0$ and the negatively oriented disk $a t - b x - c y - d z > 0$. Given a disk $D$, we denote its boundary circle by $\partial D$.

Scaling the coordinates of $(a, b, c, d)$ by any $\lambda>0$ corresponds to the same disk and by any $\lambda < 0$ corresponds to the opposite disk with the same boundary. Thus each disk is identified with a ray in $\mathbb{R}^{1,3}$. Note that unlike with our model of $\widehat{\mathbb{E}}^3$, there is no constraint on the sign of the $a$-coordinate.

A hyperplane $(a, b, c, d)$ has a real circle of intersection with $\mathbb{S}^2$ if and only if its normal is spacelike. If $(a, b, c, d)$ is lightlike, then the intersection of the hyperplane $(a, b, c, d)$ with $\mathbb{S}^2$ is a point. If $(a, b, c, d)$ is timelike, then the intersection is imaginary. Thus, we categorize circles and disks of $\mathbb{S}^2$ as real, point, or imaginary depending on whether the Lorentz normal is spacelike, lightlike, or timelike. From here on we think of a 4-tuple $(a, b, c, d)$ as simultaneously defining a normal vector in $\spacetime$, a halfspace incident a hyperplane in $\spacetime$, and a disk on $\mathbb{S}^2$ that may be real, point, or imaginary. 

\paragraph{Areas of disks on $\mathbb{S}^2$}
The $a$-coordinate of a real disk determines a bound on the area of the disk. When $a > 0$, the area of the disk is less than $2\pi$. When $a = 0$, the disk is a great circle and thus has area $2\pi$. When $a < 0$, the area of the disk is greater than $2\pi$. For point disks, a positive $a$-coordinate corresponds to zero area, while a negative $a$ coordinate corresponds to a covering of the entire sphere with area $4\pi$.

\paragraph{Conical caps.} 
Let $D = (a, b, c, d)$ be a real disk of area at most $2\pi$. Treat $(a, b, c, d)$ as a vector in $\spacetime$. Since $D$ is real, the vector $(a, b, c, d)$ is spacelike. Let $D^*$ denote the ray from the origin of $\spacetime$ in the direction $(a, b, c, d)$. This ray corresponds to a point of $\widehat{\mathbb{E}}^3$ which is finite when $a > 0$ (i.e. $D$ has area less than $2\pi$) and on the sphere at infinity when $a = 0$ (i.e. $D$ has area $2\pi$). When $a > 0$ the point $D^*$ in $\widehat{\mathbb{E}}^3$ is precisely the apex of the cone tangent to $\mathbb{S}^2$ at $\partial D$. When $a = 0$, $\partial D$ is a great circle and the cylinder tangent to $\mathbb{S}^2$ at $\partial D$ is parallel to the vector $(b, c, d)$. Properly, this is a cone with apex at the point of infinity in direction $(b, c, d)$ tangent to $\mathbb{S}^2$ in $\widehat{\mathbb{E}}^3$.

\paragraph{Lorentz transformations and Möbius transformations.} The isometries (i.e. Lorentz inner product preserving maps) of $\mathbb{R}^{1,3}$ that fix the origin form a group called the \term{Lorentz group}. The subgroup of the Lorentz group that fix the time direction $t$ and are orientation preserving on the space directions $x$, $y$, and $z$ are called the \term{restricted Lorentz group}. The elements of this group map the light cone to itself, or rather map $\mathbb{S}^2$ to itself. Restricted to $\mathbb{S}^2$, the restricted Lorentz group is simply the 6-dimensional orientation preserving \term{Möbius group}, $\textrm{Möb}(\mathbb{S}^2)$, on $\mathbb{S}^2$.  

\subsection{Inversive Distance}

The \term{inversive distance} is a Möbius invariant measurement defined on two disks $D_1$ and $D_2$ of $\mathbb{S}^2$. In our parametrization, the inversive distance $d(D_1, D_2)$ is simply the negative normalized Lorentz inner product between them. 
\begin{equation}
	d(D_1, D_2) = -\frac{\langle D_1, D_2\rangle_{1,3}}{|D_1|_{1,3}|D_2|_{1,3}},
\end{equation}
When $D_1$ and $D_2$ are both real disks, then the normal to their defining planes are spacelike. Thus we may normalize the coordinates of $D_1$ and $D_2$ so that $\langle D_1, D_1 \rangle_{1, 3} = \langle D_2, D_2 \rangle_{1, 3} = -1$. In this case, the inversive distance between $D_1$ and $D_2$ simplifies to 
\begin{equation}
	d(D_1, D_2) = \langle D_1, D_2 \rangle_{1, 3}
\end{equation}
We will call the coordinates of a normalized real disk its \term{de Sitter coordinates}. Given any real disk $D=(a, b, c, d)$, its de Sitter coordinates are $(\lambda a, \lambda b, \lambda c, \lambda d)$ where $\lambda$ is the normalization factor $\lambda = 1/\sqrt{b^2 + c^2 + d^2 - a^2}$. The \term{absolute inversive distance}, $|d(D_1, D_2)|$ is sometimes useful as well.  

Inverting a disk $D_1$ across its boundary to obtain $-D_1$ flips the sign of its inversive distance with any other disk, since $\langle D_1, D_2 \rangle_{1, 3} = -\langle -D_1, D_2\rangle_{1, 3}$.

The inversive distance between two disks gives us information on how they situate relative to one another. Let $D_1$ and $D_2$ be two disks and $C_1 = \partial D_1$ and $C_2 = \partial D_2$ be their boundary circles. If $|d(D_1, D_2)| > 1$ then $C_1$ and $C_2$ are disjoint. When $d(D_1, D_2) < -1$, one of the disks completely contains the other on its interior. When $d(D_1, D_2) > 1$, then either the two disks are disjoint, or intersect at an annular region. When $|d(D_1, D_2)| = 1$, then $C_1$ and $C_2$ are tangent at a point $p$. If $d(D_1, D_2) = -1$ then one of the disks contains the interior of the other on its interior. Otherwise, the two disks either meet at just the point $p$ or their intersection forms a crescent region that comes to a point at $p$. If $d(D_1, D_2)\in[-1, 1]$ then $C_1$ and $C_2$ intersect at two distinct points  $p_1$ and $p_2$ and divide $\mathbb{S}^2$ into four lune regions, each of which is bounded by an arc of $C_1$ and an arc of $C_2$. The angle of intersection $\theta$ on the interior of both $D_1$ and $D_2$ can be computed from the inversive distance via
\begin{equation}
	\cos \theta = d(D_1, D_2).
\end{equation}
We call this the \term{overlap angle} of $D_1$ and $D_2$. When $0\leq \theta \leq \pi/2$ we say that the overlap is \term{shallow} and when $0\leq \theta < \pi/2$ we say the overlap is \term{strictly shallow}. Note that two disks $D_1$ and $D_2$ have a shallow overlap if and only if $d(D_1, D_2) \in [0, 1]$ and a strictly shallow overlap if and only if $d(D_1, D_2)\in (0, 1]$. 

\paragraph{Möbius transformations and inversive distance data.}
The Möbius invariance of the inversive distance leads to several nice properties. First, there exists a Möbius transformation taking a pair of disks $(D_1, D_2)$ to a pair $(D_1', D_2')$ if and only if the inversive distances are the same between the two pairs, $d(D_1 D_2) = d(D_1', D_2')$. 

Second, given two triples of disks $(D_1, D_2, D_3)$ and $(D_1', D_2', D_3')$ such that each triple is linearly independent as rays in $\spacetime$, there exists a unique Möbius transformation $M$ taking $D_1\mapsto D_1'$, $D_2\mapsto D_2'$, and $D_3\mapsto D_3'$ if and only if the pairwise distances match, $d(D_i, D_j) = d(D_i', D_j')$ for $i, j = 1, 2, 3$, $i\neq j$. 

Finally, given two distinct disks $D_1$ and $D_2$, there is a one parameter continuous differentiable subgroup of $\textrm{Möb}(\mathbb{S}^2)$ called a \term{Möbius flow} that fixes both disks. More information on Möbius flows may be found in \cite{Bowers2017}. 
% In fact, the flow fixes every disk of the coaxial family $D_1\vee D_2$ and is bijective on every disk in $(D_1\vee D_2)^\perp$. 

\subsubsection{Planar representations of disks and inversive distances.}

Any pattern of disks on $\mathbb{S}^2$ may be taken to a pattern of disks on the extended complex plane (Riemann sphere) $\mathbb{C}\cup\{\infty\}$ by stereographic projection from the North pole $N$. The North pole itself is mapped to $\infty$. The boundary circle $\partial D$ of a disk $D$ either maps to a circle in $\mathbb{C}$ if $\partial D$ does not contain $N$ or to a line in $\mathbb{C}$, which is a circle of infinite radius in $\mathbb{C}\cup\{\infty\}$ containing the point $\infty$, if $\partial D$ does contain $N$. Each circle in $\mathbb{C}$ bounds two disks in $\mathbb{C}\cup\{\infty\}$, the first being the standard interior disk and the second being the exterior disk containing $\infty$. Topologically, both are disks in $\mathbb{C}\cup\{\infty\}$. A line in $\mathbb{C}$ (which is still a circle in $\mathbb{C}\cup\{\infty\}$) bounds two half-planes which are disks in $\mathbb{C}\cup\{\infty\}$.

A key property of stereographic projection is that it preserves inversive distances between pairs of disks. Thus we may start with a set of $n$ disks $\mathbf{D}$ on $\mathbb{S}^2$, stereographically project to a set of disks on $\mathbb{C}\cup\{\infty\}$ and then apply a Möbius transformation to $\mathbb{C}\cup\{\infty\}$ to obtain a new set of disk $\mathbf{D}'$ in $\mathbb{C}\cup\{\infty\}$ while preserving all pairwise inversive distances. 

We give a partial parametrization of disks on $\mathbb{C}\cup\{\infty\}$. Let $(x, y, r)$ denote the center $(x, y)\in\mathbb{C}$ of a disk and $r\neq 0$ be its signed radius. The boundary circle of $(x, y, r)$ is the circle centered at $(x, y)$ with radius $|r|$. The disk $(x, y, r)$ is the usual interior of this circle when $r > 0$ and denotes the exterior disk when $r < 0$. Under this parametrization, the inversive distance between two disks $D_1 = (x_1, y_1, r_1)$ and $D_2 = (x_2, y_2, r_2)$ is 
\begin{equation}
d(D_1, D_2) = \frac{d_{12}^2 - r_1^2 - r_2^2}{2r_1 r_2}	
\end{equation}
where $d_{12}^2 = (x_2-x_1)^2 + (y_2-y_1)^2$ is the squared distance between their centers. By a limiting argument, the inversive distance between a finite disk $D_1$ and a disk of infinite radius $D_2$ may be shown to be
\begin{equation}
	d(D_1, D_2) = \frac{d_{12}}{2r_1}
\end{equation}
where $d_{12}$ here denotes the signed distance from the center $(x_1, y_1)$ to the boundary line in $\mathbb{C}$. The sign of this distance is positive if $(x_1, y_1)$ is not in $D_2$, negative otherwise.

The inversive distance between two disks of infinite radius $D_1$ and $D_2$ is $1$ if their boundary lines are parallel in $\mathbb{C}$ (and hence their boundary circles are tangent at $\infty$) or is the cosine angle of overlap between them $\theta_{12}$:
\begin{equation}
	d(D_1, D_2) = \cos\theta_{12}.
\end{equation}

\subsubsection{Computations with inversive distances}

Here we do a few quick computations with inversive distances that are used later.

\begin{Lemma}
	Let $b$, $c$, and $d$ be given. There exist two Lorentz-normalized disks with $b$, $c$, and $d$ as the last three coordinates. Both are real, point, or imaginary depending on whether  $b^2 + c^2 + d^2$ is greater than, equal to, or less than 1. 
\end{Lemma}
\begin{proof}
	Let $D = (a, b, c, d)$ be a normalized disk with the desired three last coordinates. Since $D$ is Lorentz-normalized $\langle D, D, \rangle_{1,3} = -1 \Rightarrow a^2 - b^2 - c^2 - d^2 = -1\Rightarrow a^2 = b^2 + c^2 + d^2 -1 \rightarrow a = \pm \sqrt{b^2 + c^2 + d^2 - 1}$. Thus there are two real normalized disks 
\end{proof}

\begin{Lemma}\label{lem:twosolutionsgivenbc}
	Let $D = (a, b, c, d)$ be a normalized disk on $\mathbb{S}^2$. Let $\delta \in \mathbb{R}$ be a desired inversive distance. Let $b', c'\in \mathbb{R}$ be given. Then there exists at most two real normalized disks $D_1'$ and $D_2'$ with $b$ and $c$ coordinates given by $b'$ and $c'$ such that $d(D, D_1') = d(D, D_2') = \delta$. 
\end{Lemma}
\begin{proof}
	Let $D = (a, b, c, d)$ and  $\delta, b', c'\in\mathbb{R}$ be given. A normalized disk $D'$ with $b$ and $c$ coordinates given by $b'$ and $c'$ and third coordinate $d'$ has coordinates $(\sqrt{(b')^2+(c')^2+(d')^2-1}, b', c', d')$. Then $d(D, D') = \langle D, D'\rangle_{1,3} = a \sqrt{(b')^2+(c')^2+(d')^2-1} - b b' - c c' - d d'$. Setting equal to $\delta$ and solving for $d'$ we obtain a quadratic equation in $d'$, thus proving the lemma. 
\end{proof}

We end with a computation in the planar representation. 

\begin{Lemma}
	Let $D_1$ and $D_2$ be the interior disks of finite radius boundary circles in $\mathbb{C}\cup\{\infty\}$. If the center of $D_2$ is contained in $D_1$ then $d(D_1, D_2) < 0$ (or vice versa).
\end{Lemma}
\begin{proof}
	Since translations and rotations of $\mathbb{C}$ are Möbius transformations we may, without loss of generality, assume that $D_1$ is centered at the origin and $D_2$ is centered at a point $(x, 0)$ on the real axis. Then the inversive distance is
	\[
	d(D_1, D_2) = \frac{x^2 - r_1^2 - r_2^2}{2r_1 r_2}.
	\] 
	Since $D_1$ and $D_2$ are interior disks, then $r_1, r_2 > 0$. If the center of $D_2$ is contained in $D_1$, then $x < r_1\Rightarrow x^2 < r_1^2\Rightarrow x^2 - r_1^2 - r_2^2 < 0 \Rightarrow d(D_1, D_2) < 0$. The same argument holds if the center of $D_1$ is contained in $D_2$. 
\end{proof}

\subsection{Coaxial families}
\label{sec:coaxialfamilies}

Let $D_1$ and $D_2$ be two disks on $\mathbb{S}^2$ not sharing the same boundary circle. The set $\Pi$ of linear combinations of $D_1$ and $D_2$ forms a 2-dimensional plane in $\mathbb{R}^{1,3}$. The intersection of this plane with $\deSitter$ gives a family of disks, which is called the \term{coaxial family} of $D_1$ and $D_2$, denoted $D_1\vee D_2$. The set of directions in $\mathbb{R}^{1, 3}$ orthogonal to both $D_1$ and $D_2$ spans a second 2D plane $\Pi^\perp$ through the origin of $\mathbb{R}^{1, 3}$. The intersection of this plane with $\deSitter$ also forms a coaxial family, which we call the \term{orthogonal coaxial family}, denoted $(D_1\vee D_2)^\perp$. A coaxial family and its orthogonal, taken together, is called an \term{Appollonian family} (because such families were studied by Appollonius of Perga). Every disk in a coaxial family meets every disk in the orthogonal family at a right angle, and thus if $D\in D_1\vee D_2$ and $D'\in (D_1\vee D_2)^\perp$ then $d(D, D') = 0$. Similarly, if $d(D, D_1) = 0$ and $d(D, D_2) = 0$, then $D\in (D_1\vee D_2)^\perp$.  

The intersection of $\Pi$ with $\widehat{\mathbb{E}}^3$ is the line $L$ connecting conical caps $D_1^*$ and $D_2^*$. If one, say $D_1$ is a great circle, then it is the line through the conical cap of $D_2$ parallel to the normal of $D_1$. If both $D_1$ and $D_2$ are great circles, then $\Pi$ does not meet $\mathbb{E}^3$. Here it represents the oriented projective line at infinity (which is a circle) bounding the plane through the origin of $\mathbb{E}^3$ whose normal is the cross product of the normals of $D_1$ and $D_2$. In this case $\partial D_1$ and $\partial D_2$ intersect at antipodal points on the sphere $\mathbb{S}^2$ and the coaxial family is the set of disks whose boundary pass through these two antipodal points. 

\paragraph{Classification of coaxial families.}
Coaxial families come in three types depending on whether the line $L$ intersects, is tangent to, or is disjoint from the sphere $\mathbb{S}^2$, which we call \term{hyperbolic}, \term{parabolic}, and \term{elliptic} respectively. The orthogonal family of a hyperbolic family is elliptic (and vice versa), while the orthogonal to a parabolic family is also parabolic. The type of a coaxial family may be determined by measuring the absolute inversive distance between any two distinct members of the family. Let $D_1$ and $D_2$ be disks. If $|d(D_1, D_2)| > 1$, then $D_1\vee D_2$ is hyperbolic; if $|d(D_1, D_2)| = 1$, then parabolic; and if $|d(D_1, D_2)| < 1$, then elliptic. Thus the boundary circles of any two members of a hyperbolic family are disjoint, of a parabolic family are tangent, and of an elliptic family intersect at two points. 

\paragraph{Möbius transformations of coaxial families.} It is a standard fact of inversive geometry that there is a Möbius transformation $M$ taking one coaxial family $F_1$ to another coaxial family $F_2$ if and only if $F_1$ and $F_2$ are the same type. Furthermore, if $M$ takes $F_1$ to $F_2$, then $M$ also takes the orthogonal families $(F_1)^\perp$ to $(F_2)^\perp$. This is useful because it allows us to take any coaxial family to certain standard coaxial families that can be used to simplify computations. For example, any hyperbolic family may be taken by a Möbius transformation to the set of latitude disks on $\mathbb{S}^2$. Similarly, any elliptic family may be taken to the set of longitude lines. A parabolic family make be taken via stereographic projection to the set of lines/half-planes in $\mathbb{C}$. Taking a family to some standard view often simplifies proofs, and we make extensive use of this below.

\subsection{Orthodisks, orthocircles, and c-planes}

Let $D_1 = (a_1, b_1, c_1, d_1)$, $D_2 = (a_2, b_2, c_2, d_2)$, and $D_3 = (a_3, b_3, c_3, d_3)$ be three linearly independent disks given on $\mathbb{S}^2$. Taken as rays in $\spacetime$ there is a unique hyperplane $\Pi$ through the origin of $\spacetime$ containing $D_1$, $D_2$, and $D_3$. It is an exercise to show that the Lorentz normal to this hyperplane is given by\begin{equation}
	C_\perp = 
	\left( 
	\begin{vmatrix}
		b_1 & c_1 & d_1 \\
		b_2 & c_2 & d_2 \\ 
		b_3 & c_3 & d_3
	\end{vmatrix},
	\begin{vmatrix}
		a_1 & c_1 & d_1 \\
		a_2 & c_2 & d_2 \\ 
		a_3 & c_3 & d_3
	\end{vmatrix},
	-\begin{vmatrix}
		a_1 & b_1 & d_1 \\
		a_2 & b_2 & d_2 \\ 
		a_3 & b_3 & d_3
	\end{vmatrix},
	\begin{vmatrix}
		a_1 & b_1 & c_1 \\
		a_2 & b_2 & c_2 \\ 
		a_3 & b_3 & c_3
	\end{vmatrix}
	\right).
\end{equation}
The intersection of this hyperplane with with $\widehat{\mathbb{E}}^3$ corresponds either to a plane in $\mathbb{E}^3$ together with its circle at infinity, or the sphere at infinity itself in $\widehat{\mathbb{E}}^3$. Since $\Pi$ is a hyperplane through the origin of $\spacetime$, it corresponds to a real, point, or imaginary circle depending on whether its intersection with $\mathbb{S}^2$ is real, point, or imaginary. We call this circle the \term{orthocircle} of the triple $(D_1, D_2, D_3)$. The two disks bounded by this circle, which correspond to the intersections of the two half-spaces defined by $\Pi$ with $\mathbb{S}^2$ we call \term{orthodisks}. The half-space in the positive direction of normal $C_\perp$ is the \term{positively oriented orthodisk} $O^+$ and that in the negative direction of $C_\perp$ is the \term{negatively oriented orthodisk} $O^-$. Any even permutation of the triple $(D_1, D_2, D_3)$ defines the same positive and negative orthodisks, whereas an odd permutation of the triple swaps the positive for the negative. 

We classify triples of disks by whether their orthocircle is an imaginary, point, or real circle. A triple $(D_1, D_2, D_3)$ is \term{elliptic} if its orthocircle is imaginary. It is \term{parabolic} if its orthocircle is a point. It is \term{hyperbolic} if its orthocircle is real\footnote{When an orthodisk is hyperbolic its interior may be viewed as a Poincaré disk model of the hyperbolic plane, and its defining disks as hyperbolic lines.}. 

Now, let $D_4$ be any other point of $\Pi$ in $\spacetime$. Since every vector from the origin to a point in $\Pi$ is Lorentz-orthgonal to the normal $C_\perp$ we have that $d(C_\perp, D_4) = 0$ and conversely if $D_4$ is any disk such that $d(C_\perp, D_4) = 0$, then $D_4$ is contained in $\Pi$. The intersection of $\Pi$ with the $\deSitter$ gives us the set of real disks meeting $C_\perp$ Lorentz-orthogonally. We call this set a \term{c-plane} of disks. Testing whether a disk $D_4$ is contained in the c-plane defined by a triple $(D_1, D_2, D_3)$ follows the usual determinant test. The sign of the following determinant
\begin{equation}\label{eq:determinanttest}
det(O^+, D_4) = 
	\begin{vmatrix}
		a_4 & b_4 & c_4 & d_4 \\ 
		a_1 & b_1 & c_1 & d_1 \\
		a_2 & b_2 & c_2 & d_2 \\
		a_3 & b_3 & c_3 & d_3 
	\end{vmatrix}
\end{equation}
determines when $D_4$ is in the c-plane defined by $D_1$, $D_2$, and $D_3$.
When $det(O^+, D_4) = 0$, $D_4$ is in the c-plane of $(D_1, D_2, D_3)$. Otherwise, the sign of the determinant determines on which side of the hyperplane $\Pi$ $D_4$ lies. We use this in section~\ref{sec:circlepolyhedra} to define the notion of convexity for circle polyhedra. As with the triples themselves, we classify the c-planes defined by a triple as hyperbolic, parabolic, or elliptic. 

A result used at least as far back as Thurston's {\em Notes} \cite{Thurston:1980} is that three circles that mutually overlap or are tangent have a common real orthocircle if and only if the sum of their overlap angles is less than $\pi$. When equal to $\pi$, the three circles are coincident at a single point which corresponds to the parabolic orthocircle of the three. An elegant geometric proof of this fact may be obtained by treating the interior of the orthocircle as the Poincaré disk model of the hyperbolic plane. The arcs of the three circles restricted to the interior of the orthocircle become hyperbolic lines and since the three circles intersect, they describe a hyperbolic triangle. The interior angles of the triangle are the angles of overlap between the circles, and thus sum to less than $\pi$. This is summarized in the following lemma. 

\begin{Lemma}
	Let $D_1, D_2, D_3$ be three disks such that $d(D_i, D_j)\in [0, 1]$ for $i, j = 1\dots 3$, $i\neq j$ and let $O_+$ denote the positively oriented orthodisk of $(D_1, D_2, D_3)$. Let $\theta_{ij}$ denote the angle of overlap between disks $D_i$ and $D_j$, which is $0$ if $\partial D_i$ and $\partial D_j$ are tangent. Let $\theta = \theta_{12} + \theta_{13} + \theta_{23}$. Then $O_+$ is hyperbolic when $\theta < \pi$, is parabolic when $\theta = \pi$, and is elliptic when $\theta > \pi$. 
\end{Lemma}

\paragraph{Möbius transformations of c-planes} A hyperbolic c-plane corresponds to a hyperplane that intersects the interior of the light cone in $\spacetime$ and thus has a spacelike normal; a parabolic c-plane corresponds to a hyperplane tangent to the light cone with a lightlike normal; and an elliptic c-plane corresponds to a hyperplane that does not meet the interior of the light cone and has a timelike normal. Since Lorentz transformations preserve  the spacelike, lightlike, or timelikeness of vectors, this has implications for Möbius transformations between c-planes. Given two c-planes $\Pi_1$ and $\Pi_2$ there exists a Möbius transformation mapping $\Pi_1$ bijectively to $\Pi_2$ if and only if the two c-planes are the same type (i.e. both hyperbolic, or both parabolic, or both elliptic).

\subsection{Properties of small disk arrangements}

Here we include a few basic facts about triples and quadruples of disks.

\paragraph{Parabolic triples}
Let $(D_1, D_2, D_3)$ be a parabolic triple of disks and let $p$ be the common point of intersection $p = \partial D_1 \cap \partial D_2 \cap \partial D_3$ as in figure~\ref{fig:paraboliccases}. The point $p$ may (a) be on the interior of the union of the disks, (b) be on the boundary of the union of the disks next to a vanishing angle, or (c) be on the boundary of the union of the disks next to a positive angle. We now show that if all pairwise inversive distances of the triple fall in $[0, 1)$ we have case (a) and case (b) occurs if and only if two disks are tangent and meet the third disk at an overlap of $\pi/2$ (i.e. the inversive distances are $0$, $0$, and $1$).

\begin{figure}[hbt]
  \centering
  \begin{minipage}[b]{0.3\textwidth}
  	\includegraphics[width=\textwidth]{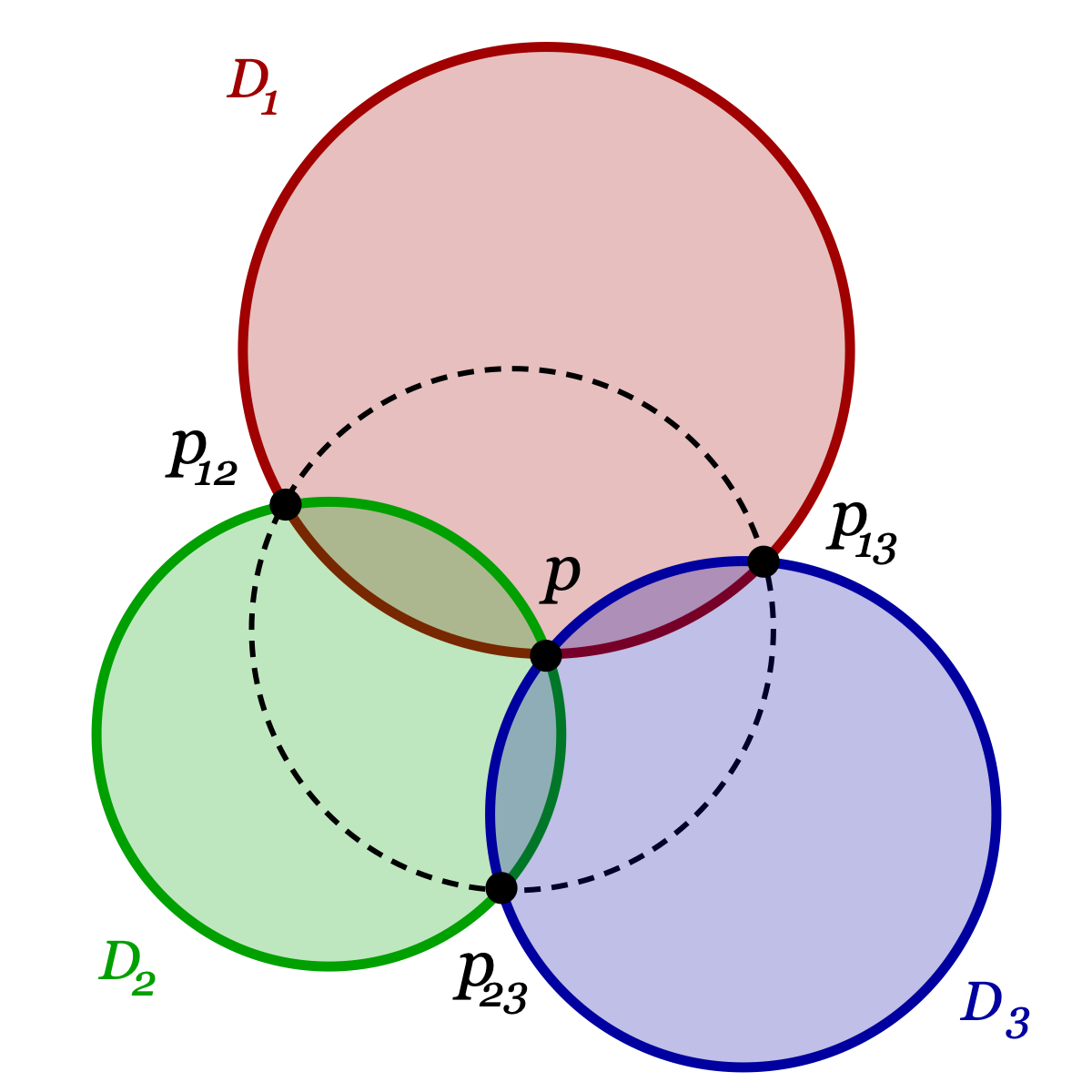}
	\subcaption{}
  \end{minipage}
  \begin{minipage}[b]{0.3\textwidth}
	  \includegraphics[width=\textwidth]{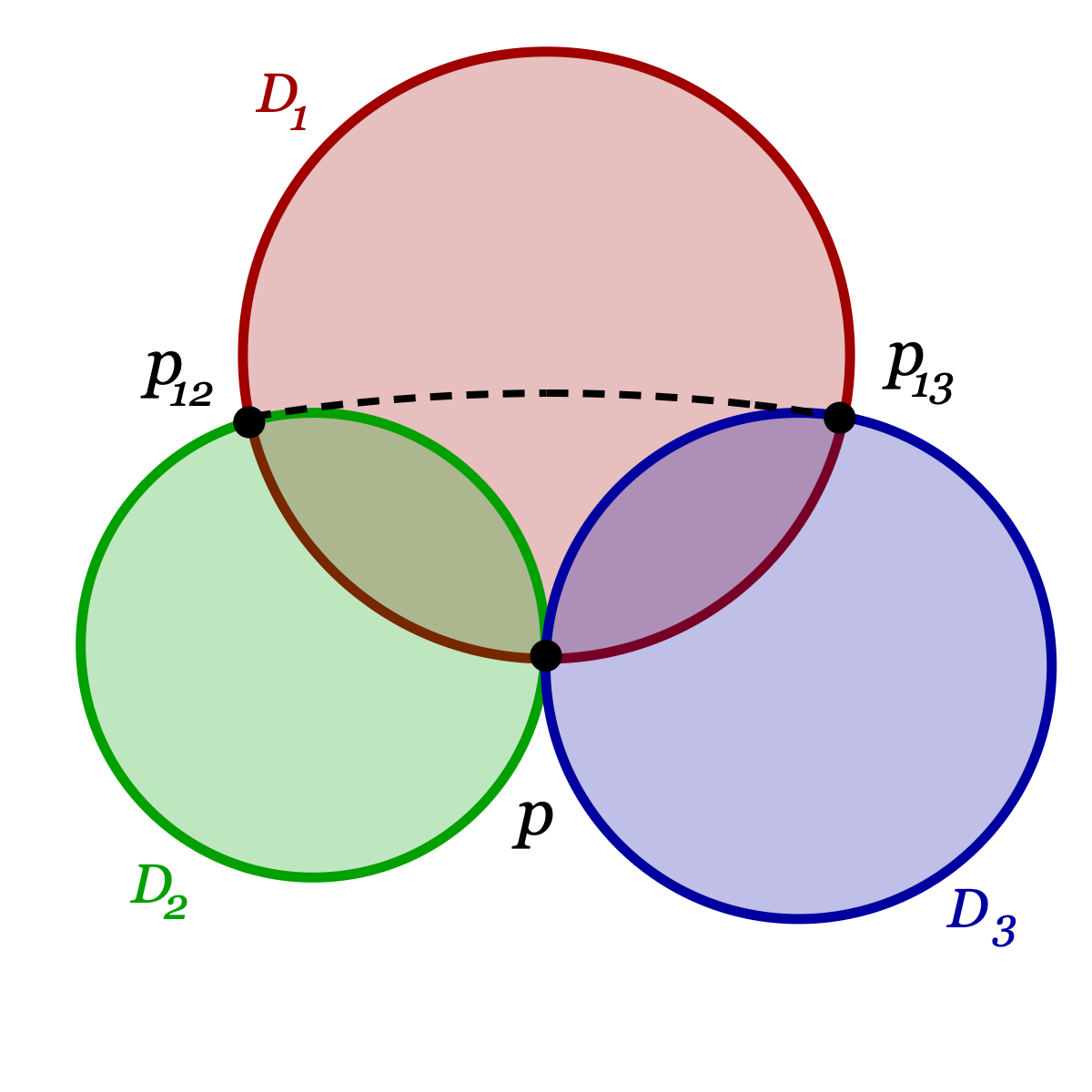}
	  \subcaption{}
  \end{minipage}
  \begin{minipage}[b]{0.3\textwidth}
  	\includegraphics[width=\textwidth]{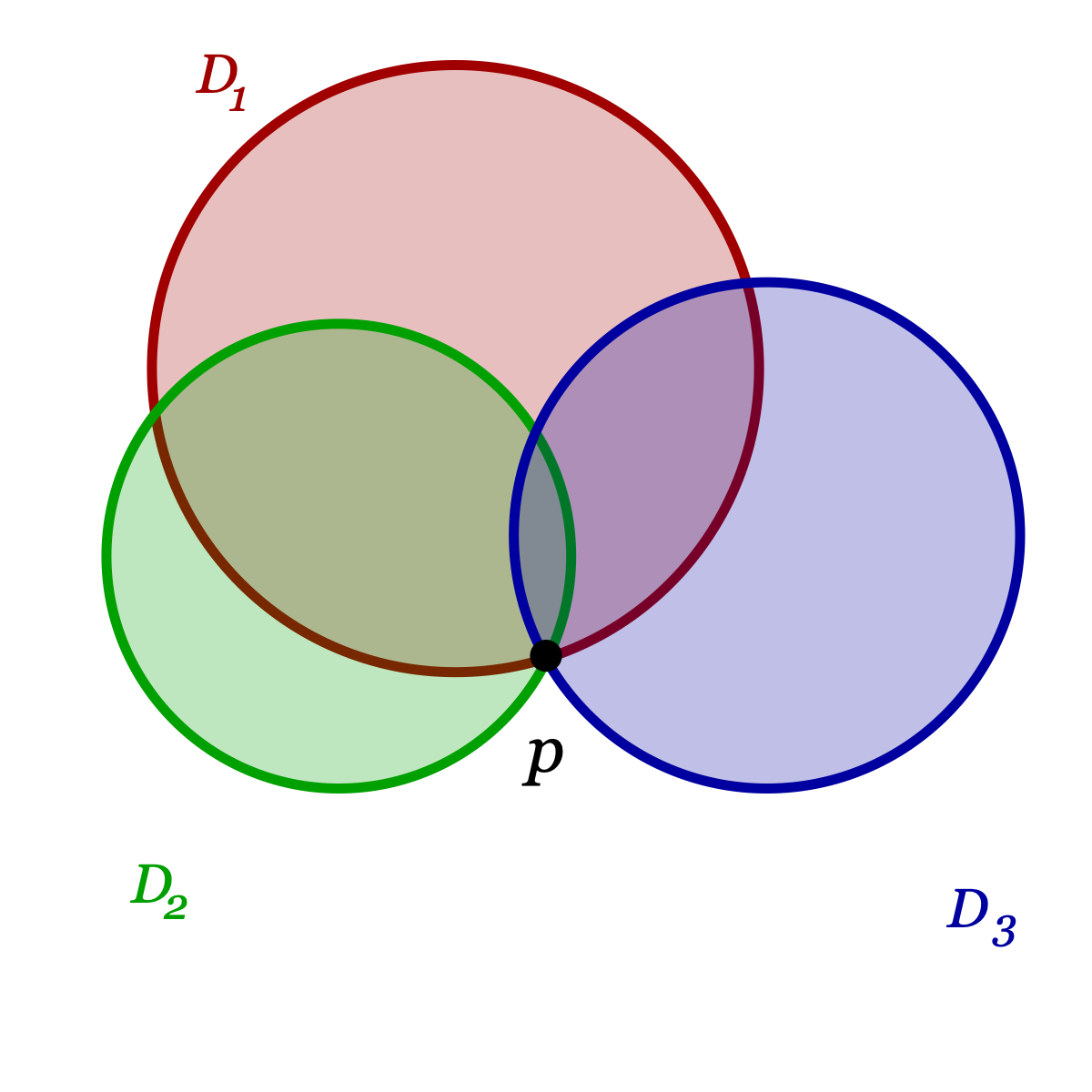}	
  	\subcaption{}
  \end{minipage}
  \caption{The three cases of parabolic triples. Case (b) occur if and only if the inversive distance between $D_1$ and both $D_2$ and $D_3$ is 0 and the other two are tangent at $p$. Case three cannot occur with all inversive distances in $[0, 1]$.}
  \label{fig:paraboliccases}
\end{figure}

\begin{Lemma}\label{lem:paraboliccases}
	Let $D_1$, $D_2$, and $D_3$ be three disks that meet at a point $p$ such that $d(D_i, D_j)\in [0,1]$ for $i\neq j$, $i, j = 1, 2, 3$. Then either $p$ is on the interior of $D_1\cup D_2 \cup D_3$ and no two of the disks are tangent or $p$ is the point of tangency between two of the disks which both meet the third disk at a right angle. 
\end{Lemma}
\begin{proof}
Stereographically project the disks to the extended complex plane $\widehat{\mathbb{C}}$. Then apply a Möbius transformation taking $p$ to $\infty$. The result is that all three disks become lines in $\mathbb{C}$ as in figure~\ref{fig:parabolicproof}. If it exists, let $p_{ij}$ denote the point of intersection between $D_i$ and $D_j$ that is not $p$ and $\theta_{ij}$ denote the oriented angle of overlap between the two disks at that point (which is 0 when that point $p_{ij} = p$ in case (b)). Recall that the inversive distance is the angle of overlap between the disks, and stereographic projections and Möbius transformations are conformal, so $d(D_i, D_j) = \cos\theta_{ij}$. In case (a) the three angles are equal to the interior angle of the triangle with vertices $p_{12}$, $p_{13}$, and $p_{23}$ and thus all inversive distances are in $[0, 1]$ precisely when this triangle is non-obtuse. For case (b), without loss of generality assume that $p$ is the point of tangency between $D_2$ and $D_3$ and thus $p_{23} = p = \infty$. In this case $\theta_{12} + \theta_{13} = \pi$ and thus either both angles are $\pi / 2$ and thus both inversive distances $d(D_1, D_2) = d(D_1, D_3) = 0$ or one of the angles is greater than $\pi/2$ and its corresponding inversive distance is less than 0. Finally, in case (c) without loss of generality assume $p_{23}$ is on the interior of $D_1$. Then the angles $\theta_{12}$ and $\theta_{13}$ are exterior angles along the same supporting line of a triangle and therefore $\theta_{12} + \theta_{13} > \pi$. Thus at least one of them is greater than $\pi/2$ and its corresponding inversive distance is therefore less than $0$. 
\begin{figure}[hbt]
  \centering
  \begin{minipage}[b]{0.3\textwidth}
  	\includegraphics[width=\textwidth]{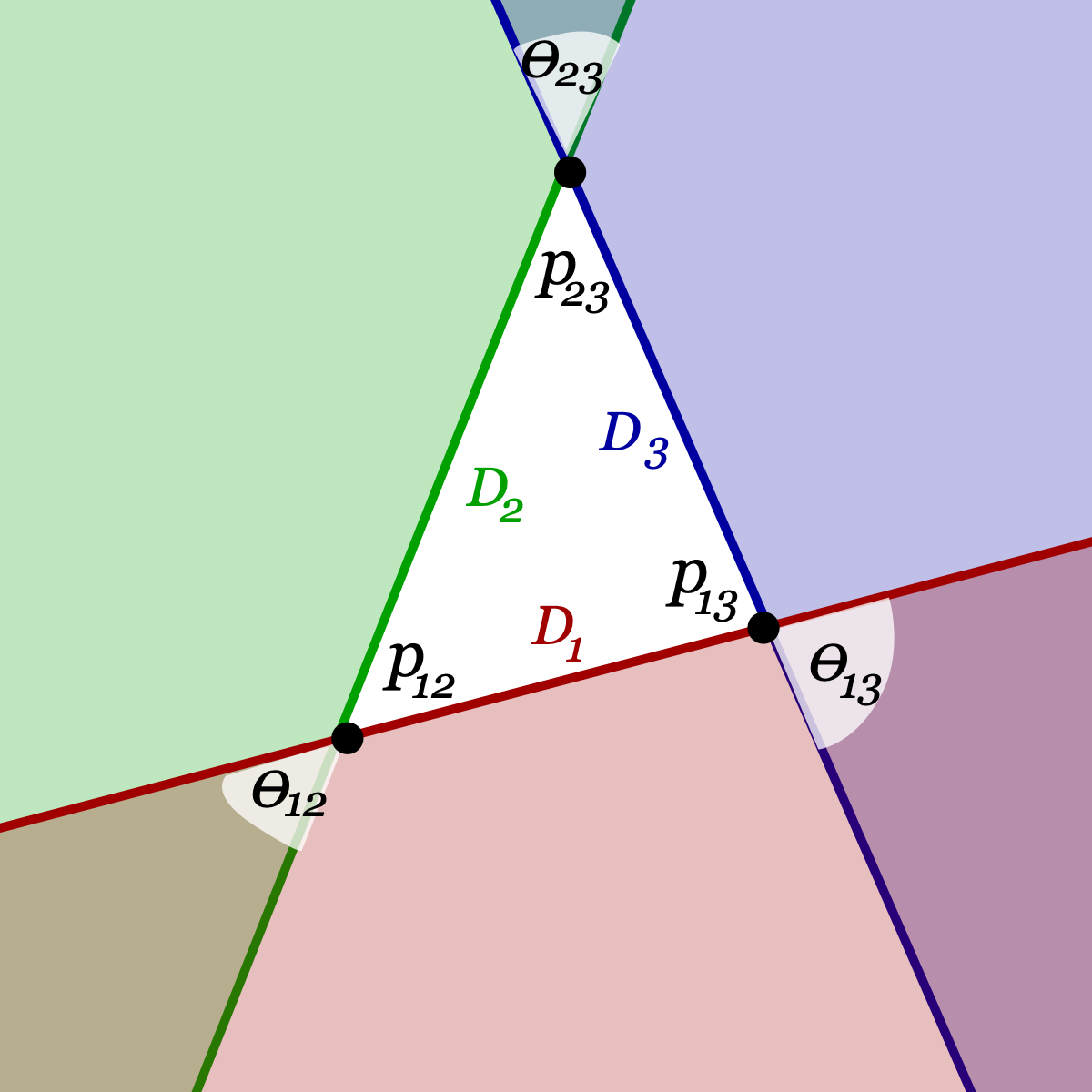}
	\subcaption{}
  \end{minipage}
  \begin{minipage}[b]{0.3\textwidth}
	  \includegraphics[width=\textwidth]{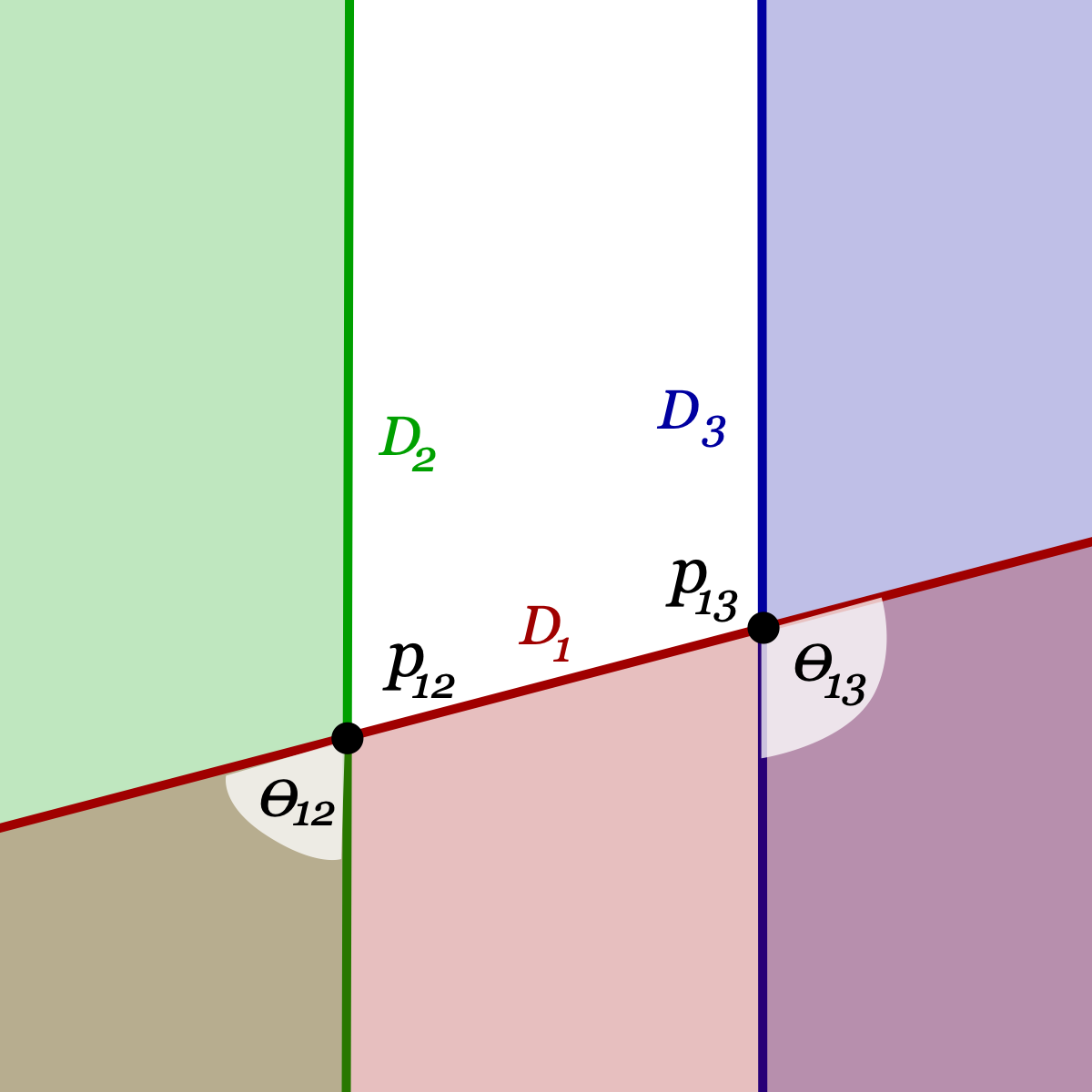}
	  \subcaption{}
  \end{minipage}
  \begin{minipage}[b]{0.3\textwidth}
  	\includegraphics[width=\textwidth]{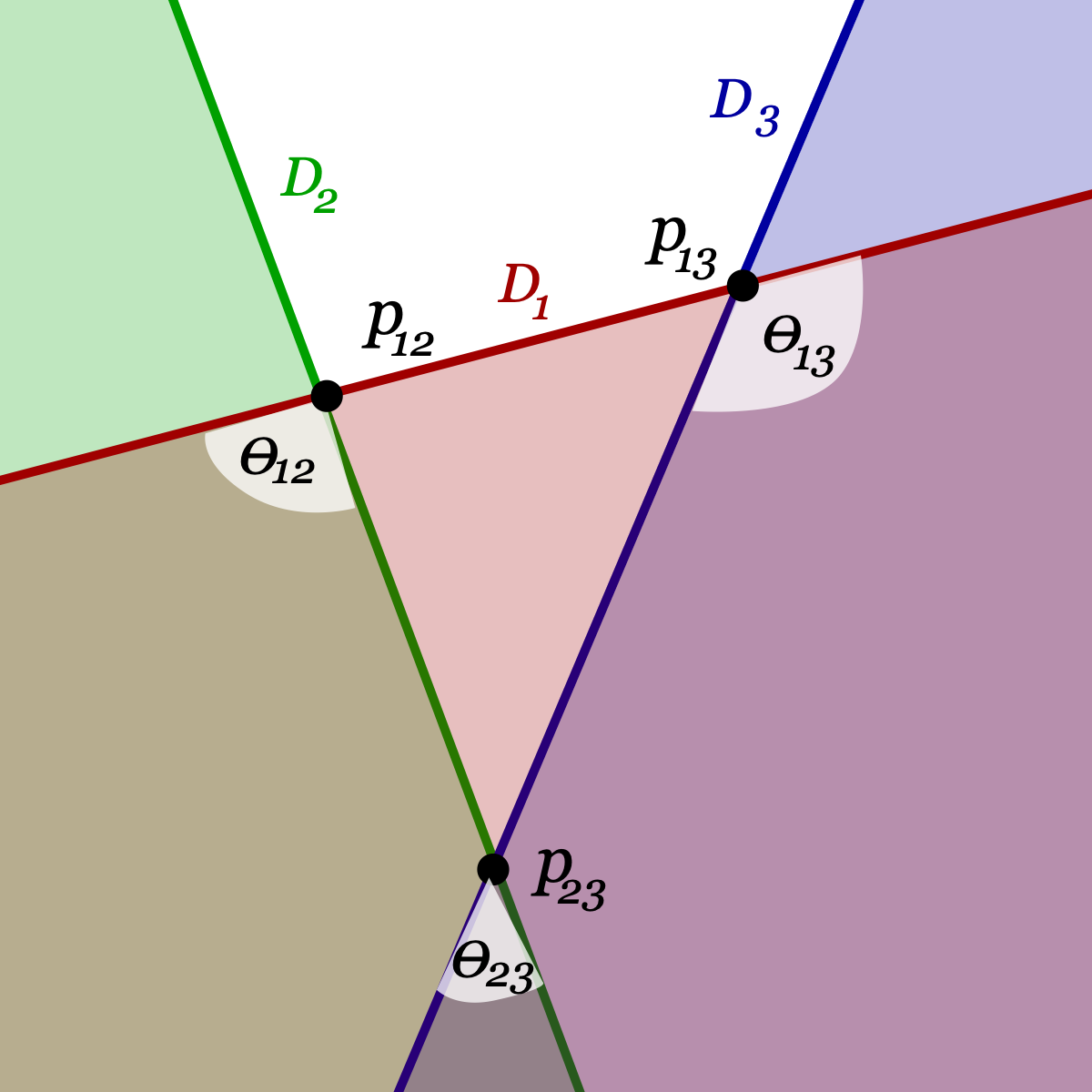}	
  	\subcaption{}
  \end{minipage}
  \caption{The setup for the proof of lemma~\ref{lem:paraboliccases}. Moving the common point of the three circles to $\infty$ by a Möbius transformation turns the boundary disks into half-planes. (a) and (b) are possible with shallow overlaps and show that the sum of the overlap angles is $\pi$. (c) is not possible with shallow overlaps. }
  \label{fig:parabolicproof}
\end{figure}
\end{proof}

An immediate corollary of the proof above follows.
\begin{Corollary}\label{cor:tripleanglesum}
	If $D_1$, $D_2$, and $D_3$ are three disks with mutually shallow overlaps, then the sum of their overlaps is $\pi$. 
\end{Corollary}

We now consider adding a fourth disk. 

\begin{Lemma}\label{lem:parabolic4diskconfig}
The only configurations of four disks contained in a parabolic plane that have only shallow overlaps are configurations comprised of two pairs of tangent disks which are mutually orthogonal. 
\end{Lemma}
\begin{proof}
	We begin the same as in the proof of lemma~\ref{lem:paraboliccases} by stereographically projecting all disks to the sphere and moving the common point $p$ to $\infty$ by a Möbius transformation. The statement is now a statement about half-planes in $\mathbb{C}$: the only configuration of half-planes that all have mutually shallow overlap angles are the configurations of two pairs of disjoint half-planes with parallel boundary lines that meet each other orthogonally. Let $l_i$ denote the boundary line and $h_i$ denote the half-plane corresponding to each disk $D_i$.
	
	We first analyze the case where there is at least one tangency. Without loss of generality, assume $D_2$ and $D_3$ are tangent. Then $l_2$ and $l_3$ are parallel and by lemma~\ref{lem:paraboliccases} since all overlaps are shallow, both $l_1$ and $l_4$ meet $l_2$ and $l_3$ orthogonally. Then $l_1$ and $l_4$ are parallel so $\partial D_1$ and $\partial D_4$ are tangent and thus $d(D_1, D_4) \in \{-1, 1\}$. 
	
	Now for contradiction assume no tangencies occur. Since all disks overlap at a real angle each of the lines $l_1\dots l_4$ intersect. The intersection of the complements of $h_1$, $h_2$, and $h_3$ is a triangle as in (a) of figure~\ref{fig:parabolicproof}. By shallowness, this triangle is acute. 
	
	Now consider any such triangle made by three mutually intersecting lines and choose an orientation for each line. By convention let the left-hand side of the orientated line be its half-space. Say that two of the lines \term{agree} at a corner of the triangle if there orientation is consistent along the boundary of the triangle, either counter-clockwise or clockwise. Note that when two lines agree, the oriented angle over overlap between their half-spaces equals the angle of the triangle. When the two orientations disagree, the angle of overlap between the half-spaces is $\pi$ minus the angle formed on the interior of the triangle at each corner. We observe that either all three corners of the triangle agree or exactly two disagree with any orientation of the lines we choose.
	
	Now take the half-plane orientations induced by our disks.  Because our half-plane overlaps are shallow it is not possible that two corners disagree, because a corner that disagrees with our orientation must have an angle of at least $\pi/2$ and no triangle has two angles both greater than or equal to $\pi/2$. Thus any triangle formed by any three shallow overlapping half-planes must agree with the orientation of those half-planes, which means that the orientation of the lines must be clockwise or counter-clockwise around the triangle.
	
	Now consider adding the fourth half-plane $h_4$ to our arrangement of $h_1$, $h_2$, and $h_3$. We observe that either $l_4$ does not meet any of the corners of the triangle formed by $l_1$, $l_2$, and $l_3$, or it meets exactly one corner. In the first case any three of the lines defines a triangle and by the discussion above, the orientations on the lines induced by $h_1\dots h_4$ must be counter-clockwise or clockwise around each of these triangles. However, there is no orientation of four lines in general position that is consistent around each of the four triangles formed by any three of the lines. Similarly, when $l_4$ passes through a corner of the triangle formed by $l_1$, $l_2$, and $l_3$ there are three triangles formed, one of which is the union of the other two. But then again it is not possible to orient the lines so that the orientation is consistent across all three triangles. But if all overlaps were shallow, then any three of the half-planes must define a consistent orientation around the triangle formed by their boundary lines. Therefore this case is not possible, and we must have two pairs of tangent disks contained in orthogonal coaxial families.
\end{proof}

As an immediate corollary we have: 

\begin{Corollary}\label{cor:no5shallowinparabolic}
	There are no five disks whose boundary circles meet at a common point such that the pairwise overlaps between all disks are shallow.
\end{Corollary}

\paragraph{Coplanarity of four disks.} Let $D_1$, $D_2$, $D_3$, and $D_4$ be four distinct disks and assume that the pairwise overlap angles between disks in the two triples $(D_1, D_2, D_3)$ and $(D_2, D_3, D_4)$ fall between 0 and $\pi/2$. When are the four disks contained in the same c-plane? Lemma~\ref{lem:cplanarity} characterizes when all four disks are coplanar. 

\begin{Lemma}\label{lem:cplanarity}
	Let $(D_1, D_2, D_3)$ and $(D_2, D_3, D_4)$ be two triples of disks such that all pairwise inversive distances in the triples are shallow and no three of $D_1\dots D_4$ is contained within a coaxial family. 
	Let $\Pi$ be the c-plane defined by the orthodisk of $(D_1, D_2, D_3)$.
	Then:
	\begin{enumerate}
		\item If $\Pi$ is hyperbolic and $D_4$ is contained in $\Pi$ then $d(D_1, D_4) < 0$. 
		\item If $\Pi$ is parabolic and $D_4$ is contained in $\Pi$ then either $d(D_2, D_3) = d(D_1, D_4) = 1$ and all other pairwise inversive distances are 0 or $d(D_1, D_4) < 0$. 
		\item If $\Pi$ is elliptic and $D_4$ is contained in $\Pi$ then either $d(D_1, D_4) \geq 0$ and $D_1\cup D_2 \cup D_3 \cup D_4 = \mathbb{S}^2$ or $d(D_1, D_4) < 0$. Furthermore, in the first case there does not exist any other disk that has shallow overlaps with each of $D_1\dots D_4$.
	\end{enumerate}
\end{Lemma}
\begin{proof}

{\em Case 1:}
In this case, let $O^+$ be the positively oriented orthodisk defining $\Pi$. Since $\Pi$ is hyperbolic, this is a real disk. Assume $d(D_4, O^+) = 0$. Since all disks $D_1, \dots, D_4$ meet $O^+$ orthogonally, they correspond to hyperbolic half-planes in $\mathbb{H}^2$ obtained by treating the interior of $O^+$ as a Poincaré disk model of $\mathbb{H}^2$. This correspondence preserves overlap angles. Let $h_i$ denote the half-plane and $l_i$ denote the boundary line corresponding to each disk $D_i$. Let $h_i'$ denote the reflection of $h_i$ across $l_i$. The angle $\theta_{ij}$ of intersection between the half-planes $h_i$ and $h_j$ is given by $\theta_{ij} = \arccos d(D_i, D_j)$ for all $i, j = 1 \dots 4$, $i\neq j$. We now argue that it is not possible that all of these angles are between $0$ and $\pi/2$. 

By the overlap conditions on the triples $(D_1, D_2, D_3)$ and $(D_2, D_3, D_4)$ the lines $l_1$, $l_2$, and $l_3$ form a hyperbolic triangle, as do $l_2$, $l_3$, and $l_4$. Because neither triple is contained in a coaxial family, neither of the triples of lines $(l_1, l_2, l_3)$ and $(l_2, l_3, l_4)$ meet at a common point. Therefore $(l_1, l_2, l_3)$ and $(l_2, l_3, l_4)$ each define a hyperbolic triangle. Let $T_{ijk}$ denote the hyperbolic triangle formed by $l_i$, $l_j$, and $l_k$. Let $\hat{\theta}_{ij}$ denote the angle between $l_i$ and $l_j$ in the triangle $T_{ijk}$.  

We first claim that if $\theta_{ij}, \theta_{jk}, \theta_{ki}\in[0,\pi/2]$, then the interior of $T_{ijk}$ either equals $h_i\cap h_j \cap h_k$ or $h_i'\cap h_j'\cap h_k'$. Indeed if $T_{ijk} \subset h_i\cap h_j$ or $T_{ijk}\subset h_i' \cap h_j'$ then $\hat{\theta}_{ij} = \theta_{ij}$ is the angle between $l_i$ and $l_j$ in $T_{ijk}$. However, if one half-plane contains $T_{ijk}$ and the other does not, say $T_{ijk}\subset h_i$ and $T_{ijk}\subset h_j'$, then the angle of the triangle $T_{ijk}$ between $l_i$ and $l_j$ is $\hat{\theta}_{ij}=\pi - \theta_{ij}$. In the former case we say that the angle $\hat{\theta}_{ij}$ \term{agrees} with the orientations of $h_i$ and $h_j$, and in the latter \term{disagrees}. Either all angles of $T_{ijk}$ agree or exactly two disagree with the orientations of $h_i$, $h_j$, $h_k$. 

Since $0 \leq \theta_{ij} \leq \pi/2$, then $\hat{\theta}_{ij} \leq \pi/2$ if it agrees and $\hat{\theta}_{ij} \geq \pi/2$ if it disagrees with the half-plane orientations. But the sum of $\hat{\theta}_{ij} + \hat{\theta}_{jk} + \hat{\theta}_{ki} < \pi$ by hyperbolic trigonometry and thus it is not possible that two angles disagree with the orientation and therefore all must agree. Orient each line $l_i$ so that its left-hand side is the half-plane $h_i$. The property above implies that this \term{induced orientation} is consistent (either counter-clockwise or clockwise) around the triangle $T_{ijk}$.

First let us assume that $D_1$ and $D_4$ are disjoint (i.e. $d(D_1, D_4) \not\in[-1,1]$. Then $l_1$ and $l_4$ do not intersect. The triangles $T_{123}$ and $T_{234}$ share the point $p$ of intersection between $l_2$ and $l_3$ in common, and an edge of each triangle is supported by each of $l_2$ and $l_3$. This means that either $T_{123} \subset T_{234}$ or vice versa, or the two triangles form a bowtie configuration on opposite sides of $p$. In each case, using only the fact that the orientations of the hyperplanes must agree at each corner of the two triangles it is easily shown that either $h_1 \subset h_4$ or $h_4 \subset h_1$. But then $D_4 \subset D_1$ and thus $d(D_1, D_4) < 1$. 

Now assume that $D_1$ and $D_4$ do overlap (i.e. $d(D_1, D_4)\in[-1,1]$). Then $l_1$ and $l_4$ meet either at a hyperbolic point or at an ideal point. Then $l_1$, $l_2$, $l_3$, and $l_4$ are an arrangement of four lines each pair of which meet and no three of which meet at a common point (due to the assumption that no three disks are contained in the same coaxial family). This arrangement defines four hyperbolic triangles by any choice of three of the lines. By the discussion above, the half-planes induce an orientation on the lines $l_1$, $l_2$, and $l_3$ that is consistent around $T_{123}$. For contradiction assume $d(D_1, D_4)\in[0, 1]$. Then all of the pairwise overlap angles $\theta_{ij} \in [0, \pi/2]$ for $i,j=1\dots 4$, $i\neq j$. This means that the half-plane of $h_4$ must induce an orientation of $l_4$ so that each of the triangles $T_{123}, T_{124}, T_{134},$ and $T_{234}$. However, for any placement of $l_4$ no matter what orientation is chosen at least one of the triangles is not consistently oriented. Therefore it is not possible for all of the overlap angles to be in $[0, \pi/2]$. Thus, $d(D_1, D_4) < 0$.  

{\em Case 2:} This is lemma~\ref{lem:parabolic4diskconfig}.

{\em Case 3:} Since $\Pi$ is elliptic there is a Möbius transformation  taking all of the disks in $\Pi$ to the set of disks bounded by great circles on $\mathbb{S}^2$. Such a disk is parametrized by a vector in $\spacetime$ with $0$ in the $a$-coordinate. Let $D_i = (0, b_i, c_i, d_i)$ denote the normalized coordinates of each disk. The inversive distance between two disks becomes: 
\[
d(D_i, D_j) = -b_i b_j - c_i c_j - d_i d_j = -(b_i b_j + c_i c_j + d_i d_j)
\]
which is simply the negative of the Euclidean inner product between the oriented normal vectors of the planes in $\mathbb{R}^3$ supporting each disk (with the normal pointed toward the interior of the disk). Since all great circles meet, for all $i,j=1\dots 4$ $d(D_i, D_j) \in [-1, 1]$. We first claim that $d(D_i, D_j) \geq 0$. Indeed, this is equivalent to claiming that the four normal vectors in $\mathbb{E}^3$ have negative Euclidean inner product. This is possible only if the great disks cover $\mathbb{S}^2$. Furthermore, there are no five vectors in $\mathbb{R}^3$ that all have pairwise negative Euclidean inner product.  
\end{proof}

\section{Circle polyhedra}\label{sec:circlepolyhedra}

\subsection{Abstract circle polyhedra and realizations}\label{sec:abstractcpolyhedra}

We call a 3-connected planar graph $P$ an \term{abstract polyhedron}. Let $V(P)$ and $E(P)$ denote the vertex and edge sets of $P$. Such graphs have a well defined face set $F(P)$. Let $n = |V(P)|$ and $m = |E(P)|$. If each face is a triangle, then $m=3n-6$ and we refer to $P$ as an \term{abstract triangulated polyhedron}. In the remainder, we will use the term polyhedron to mean triangulated polyhedron. 

Let $w:E(P)\rightarrow \mathbb{R}$ be an assignment of weights to each edge of $P$ representing desired inversive distances. We call the pair $(P, w)$ a \term{weighted abstract polyhedron}.  If each $w(ij)\in[0, 1]$, we call $(P, w)$ \term{shallow} and \term{strictly shallow} if no 4-cycle of edges $ij$, $jk$, $kl$, $li$ in $E(P)$ has weight 0 for each edge. %if $w(ij)\in (0, 1]$ for every $ij\in E(P)$. 

A \term{realization} of a weighted abstract polyhedron $(P, w)$ on $\mathbb{S}^2$ is a set of disks ${\bf D} = \{D_1, \dots, D_n\}$ on $\mathbb{S}^2$ such that $d(D_i, D_j) = w(ij)$ whenever $ij\in E(P)$. We call the pair $(P, \mathbf{D})$ a \term{circle polyhedron}, or \term{c-polyhedron}. As with the abstract polyhedra, $(P, w)$, we will call a circle polyhedra \term{(strictly) shallow} when it realizes a (strictly) shallow abstract polyhedron. 

\paragraph{Proper c-polyedra.} If the area of each disk of a circle polyhedron is strictly less than $2\pi$ we say it is \term{strictly proper} and if we relax this so that one disk may have area $2\pi$ we call it \term{proper}. 

\paragraph{Geodesic triangulations.}
Let $(P, \mathbf{D})$ be a proper shallow c-polyhedron. Let $G$ be a graph drawing on $\mathbb{S}^2$ given by placing a vertex $v_i$ at the center of each disk $D_i \in \mathbf{D}$ and for each edge $ij\in E(P)$ drawing the shortest path great circle arc on $\mathbb{S}^2$ from $v_i$ to $v_j$. Since $(P, \mathbf{D})$ is proper and shallow, the centers of two disks $D_i$ and $D_j$ connected by an edge $ij$ cannot be antipodal, and therefore the shortest path is well defined. Furthermore, by elementary geometry it can be shown that this path lies entirely within $D_i\cup D_j$. We call $G$ the \term{induced geodesic graph} of $(P, \mathbf{D})$. 

If no two arcs of $G$ self intersect and no vertex of $G$ lies on the interior of an arc it is not incident to in $P$, then $G$ is a triangulation of $\mathbb{S}^2$. In this case we say that $(P, \mathbf{D})$ has a \term{induced geodesic triangulation} $G$ and call $(P, \mathbf{D})$ a \term{geodesic circle polyhedron}. 

\paragraph{The Koebe-Andre'ev-Thruston theorem.}
We may now state the Koebe-Andre'ev-Thurston theorem. Our statement is similar in presentation to that in \cite{B19}.  

\begin{Theorem}[Koebe-Andre'ev-Thurston]\label{thm:kat}
	Let $(P, w)$ be a shallow abstract triangulated polyhedron, with $P$ not a tetrahedron, and let the following conditions hold.
	\begin{enumerate}
		\item If $e_1$, $e_2$, $e_3$ form a closed loop of edges with $\sum_{i=1}^3 \arccos w(e_i) \geq \pi$, then $e_1$, $e_2$, and $e_3$ bound a face of $P$. 
		\item If $e_1$, $e_2$, $e_3$, and $e_4$ form a closed loop of edges with $\sum_{i=1}^4 \arccos w(e_i) = 2\pi$ then $e_1$, $e_2$, $e_3$, and $e_4$ bound the union of two faces of $P$.
	\end{enumerate}
	Then there exists a geodesic c-polyhedron $(P, \mathbf{D})$ realizing $(P, w)$.  
\end{Theorem}

In section~\ref{sec:proofofkat} we give a novel proof of this theorem generalizing the proof of the tangency case from \cite{CG19}.

\subsection{Convex circle polyhedra}\label{sec:convexcpolyhedradef}
Let $(P, \mathbf{D})$ be a triangulated circle polyhedron and let $(D_i, D_j, D_k)$ be the disks corresponding to a face $ijk\in F(P)$. Let $O_+$ be the positively oriented orthodisk of $(D_i, D_j, D_k)$. We say that $O^+$ \term{separates} $(P, \mathbf{D})$ if and only if there exist disks $D_+$ and $D_-$ in $\mathbf{D}$ such that $det(D_+, O_+) > 0$ and $det(D_-, O_+) < 0$. A face \term{separates} $(P, \mathbf{D})$ precisely when its orthodisks do. $(P, \mathbf{D})$ is \term{convex} if and only if no face separates $(P, \mathbf{D})$.
Interpreted in $\mathbb{R}^{1,3}$ this is the usual separation of the rays of ${\bf D}$ by the unique hyperplane containing the rays $D_i$, $D_j$, and $D_k$. Under this interpretation a convex circle polyhedron is a convex polyhedral solid cone with apex at the origin in $\spacetime$. %If we normalize all of our disks to lie in the de Sitter sphere $\deSitter$, then the resulting polyhedron $(P, {\bf D})$ is a convex de Sitter polyhedron.

%\paragraph{Locked polyhedra.} A special case of convex, but not strictly convex polyhedron is when the four disks of two neighboring faces $ijk$ and $kjl$ are coplanar. The particular case we are interested in is when the outer 4-cycle of edges $(ij, jl, lk, ki)$ each have inversive distance 0 (i.e. their disks overlap by $\pi/2$, which the reader should recognize as the second condition in the Koebe-Andre'ev-Thurston theorem) and the common edge has inversive distance 1 (tangency). We call both such a polyhedron \term{locked} and the two faces $ijk$ and $kjl$ locked. We think of the union of the two locked faces as a single quadrilateral face. We abuse the terminology slightly and say that a locked triangulated circle polyhedron is strictly convex if it is convex in the usual sense and the only place it is non-strict in the usual sense (i.e. four or more coplanar disks) is for the four disks $D_i$, $D_j$, $D_k$, and $D_l$. 

%Convexity is independent of the choice of inner-product signature. Thus the convexity of the 3-cone $(P, \mathbf{D})$ in $\spacetime$ implies the convexity of the cone with the same coordinates under the Euclidean inner-product in $\mathbb{R}^4$. If we intersect this cone with the unit 3-sphere of $\mathbb{R}^4$ we obtain a 3-dimensional convex polyhedron in the unit 3-sphere $\mathbb{S}^3$.

\paragraph{The configuration space of circle polyhedra.}
Given an abstract triangulated polyhedron $P$ a \term{configuration} of $P$ is a set of $n$ disks ${\bf D}$, one for each vertex in $P$. We think of $\mathbf{D}$ as a point in $\mathbb{R}^{4n}$ and call it the \term{configuration space} of $P$. Let ${\bf D}(t)$ be a continuous path in $\mathbb{R}^{4n}$ on some interval $[a, b]$ such that ${\bf D}(a)$ is strictly convex and ${\bf D}(b)$ is strictly non-convex. Both strict convexity and strict non-convexity are open conditions determined by the signs of determinants of each face in $P$ with all other vertices in $P$ (see (~\ref{eq:determinanttest})). Then there is an open ball around ${\bf D}(a)$ in $\mathbb{R}^{4n}$ such that all configurations in the ball are strictly convex and all disks are real and distinct and an open ball around ${\bf D}(b)$ in which all configurations are strictly non-convex and all disks are real and distinct. Thus, there must be some intermediate time $a < t' < b$ at which either a disk $D_i$ contracts to a point disk, or two neighboring disks coincide, or $(P, {\bf D}(t'))$ is convex, but not strictly so. In the third case, the associated polyhedron in $\mathbb{S}^3$ is no longer strictly convex, which occurs when two faces become co-planar. Thus in $(P, {\bf D}(t'))$ we have two faces $ijk$ and $kjl$ incident an edge $jk$ such that all four disks $D_i$, $D_j$, $D_k$, and $D_l$ are coplanar. From this discussion we immediately have:
\begin{Lemma}\label{lem:maintainconvexity1}
	Let $(P, {\bf D}(t))$ be a path in $\mathbb{R}^{4n}$ defined for a (possibly bi-infinite) interval $t\in I$ such that there is a time $t'\in I$ where $(P, {\bf D}(t'))$ is strictly convex. If for all $t\in I$ we have that
	\begin{enumerate}
		\item $D_i(t)$ is a real disk for all $i\in V(P)$,
		\item $D_i(t) \neq D_j(t)$ for all $ij\in E(P)$, and 
		\item if $ijk, kjl\in F(P)$ are faces sharing a common edge $jk$ then $D_i(t)$, $D_j(t)$, $D_k(t)$, and $D_l(t)$ are not co-planar,
	\end{enumerate}
	then $(P, {\bf D}(t))$ is strictly convex for all $t\in I$.
\end{Lemma}
We call a path $(P, {\bf D}(t))$ a \term{motion} of the circle polyhedron $(P, {\bf D}(t))$. Note that a motion does not necessarily preserve inversive distances. We investigate the properties of motions in section~\ref{sec:motions} after deriving some basic properties of circle polyhedra in the next section.

\paragraph{Conical cap polyhedra.} Let $(P, \mathbf{D})$ be a triangulated circle polyhedron on $\mathbb{S}^2$, none of which has area greater than or equal to $2\pi$ (equivalently, the $a$ coordinate of each disk is greater than 0). Construct a Euclidean polyhedron by placing placing each vertex $i\in V(P)$ at the conical cap $D_i^*$ and realizing each edge $ij\in E(P)$ as the line segment $D_i^* D_j^*$. The resulting triangulated Euclidean polyhedron is the \term{(Euclidean) conical cap polyhedron} of $(P, {\bf D})$. The following lemma relates the convexity of the cap polyhedron with that of $(P, {\bf D})$.

\begin{Lemma}\label{lem:convexcappolyhedron}
	Let $(P, {\bf D})$ be a strictly proper triangulated circle polyhedron on $\mathbb{S}^2$, and let $\widehat{P}$ denote its conical cap polyhedron in $\mathbb{E}^3$. Then $(P, \mathbf{D})$ is a convex circle polyhedron if and only if $\widehat{P}$ is a convex Euclidean polyhedron. 
\end{Lemma}
\begin{proof}
Let $D_i = (a_i, b_i, c_i, d_i)$, $D_j = (a_j, b_j, c_j, d_j)$ and $D_k = (a_k, b_k, c_k, d_k)$ be the disks of ${\bf D}$ corresponding to a face $ijk\in F(P)$. Since no disk has $2\pi$ area or more, $a_i > 0$, $a_j > 0$, and $a_k > 0$. Thus we may scale the coefficients of each disk to obtain new coordinates: $D_i = (1, b_i/a_i, c_i/a_i, d_i/a_i)$, $D_j = (1, b_j/a_j, c_j/a_j, d_j/a_j)$, and $D_k = (1, b_k/a_k, c_k/a_k, d_k/a_k)$ without flipping the orientation of any disks. Let $O_+$ be the positively oriented orthodisk of $ijk$ and let $D = (1, b/a, c/a, d/a)$ be another disk in $(P, \mathbf{D})$ also scaled so that the first coordinate is 1. From equation~\ref{eq:determinanttest} we have that 
\begin{equation}\label{eq:capdet}
det(O_+, D) = \frac{1}{a a_i a_j a_k}
	\begin{vmatrix}
		1 & b/a & c/a & d/a \\ 
		1 & b_i/a_i & c_i/a_i & d_i/a_i \\
		1 & b_j/a_j & c_j/a_j & d_j/a_j \\
		1 & b_k/a_k & c_k/a_k & d_k/a_k 
	\end{vmatrix}
\end{equation}
The determinant on the right hand side of (\ref{eq:capdet}) is the standard determinant whose sign is used to determine whether the point $(b/a, c/a, d/a)$ lies on the positively oriented side (determinant is positive), negatively oriented side (determinant is negative), or inside (determinant is zero) the oriented Euclidean plane meeting $D_i^*$, $D_j^*$, and $D_k^*$. Since $a, a_i, a_j, a_k > 0$, the sign of this determinant is the same as $det(O_+, D)$. Thus if $D'$ is another disk in $\mathbf{D}$, $D^*$ and $(D')^*$ lie on the same side of the plane meeting $D_i^*$, $D_j^*$, and $D_k^*$ if and only if $det(O_+, D)$ and $det(O_+, D')$ have the same sign. Therefore no plane supporting any face of the cap polyhedron separates any other vertices of the polyhedron, which is thus convex.
\end{proof}

\paragraph{Extending to large disks.} Lemma~\ref{lem:convexcappolyhedron} shows that if all disks of a strictly convex circle polyhedron $(P, {\bf D})$ have area less than $2\pi$, then the resulting cap polyhedron $\widehat{P}$ is convex. We may extend the notion of conical caps to disks that have area greater than $2\pi$. Indeed for such a disk $D$ we will simply take its conical cap to be the conical cap of the disk opposite it along its boundary $\partial D$. In other words $D^* = (-D)^*$. 
Now, consider the determinant used in (\ref{eq:capdet}) and a triangle $ijk\in F(P)$. Let $O_+$ denote the orthodisk of $(D_i, D_j, D_k)$. Let $\Pi$ be the hyperplane passing through the origin of $\spacetime$ containing the rays $D_i^*$, $D_j^*$, $D_k^*$. This corresponds to a plane in $\mathbb{E}^3$ passing through the conical cap points that are the intersections of the rays $D_i^*$, $D_j^*$, $D_k^*$ with $\mathbb{E}^3$. Now let $D$ and $D'$ be two other disks in ${\bf D}$. By convexity, the rays $D^*$ and $(D')^*$ lie on the same side of $\Pi$. If both have area less than $2\pi$, then both have a positive $a$-coordinate, and thus the corresponding conical cap points in $\mathbb{E}^3$ lie on the same side of the Euclidean plane corresponding to $\Pi$. However, if one of the disks, say $D$ has area greater than $2\pi$, then its $a$-coordinate is negative. Thus by equation (\ref{eq:capdet}) the determinant on the right hand side of the equation has a flipped sign, which means that in $\mathbb{E}^3$, the conical cap points corresponding to $D^*$ and $(D')^*$ lie on opposite sides of $\Pi$. Furthermore, $D$ and $-D$ lie on opposite sides of $\Pi$ and $-D$ lies on the same side of $\Pi$ as the conical cap in $\mathbb{E}^3$ (since it is obtained by a positive scaling factor). Therefore $D^*$ and its corresponding conical cap point lie on opposite sides of $\Pi$. From this discussion we see that the cap polyhedron for a strictly convex circle polyhedron which has some disks of area greater than $2\pi$ will not be convex. This is precisely because the conical caps for disks with area greater than $2\pi$ lie on the ``wrong side'' of the plane supporting any face of the polyhedron. However, this wrong side property turns out to be well behaved. We now use it to recover a convex Euclidean fan for every vertex of the polyhedron that allows us to use 3D Euclidean geometry to investigate the geometry of the rays in $\spacetime$. This construction will be crucial to our proof of infinitesimal rigidity of convex circle polyhedra.  

\paragraph{The convex Euclidean triangle fan for a vertex.} 
Consider some vertex $0$ and its neighbors $1\dots k$ given in counter clockwise order. Let $v_i$ denote the conical cap point of $D_i$. The triangles $v_0 v_1 v_2, v_0 v_2 v_3, \dots, v_0 v_k v_1$ form a triangle fan in the cap polyhedron around $v_0$. When all disks have area less than $2\pi$, this fan is convex and lies on a convex polyhedral cone with apex $v_0$. However, if $D_1\dots D_k$ are a mix of disks with area less than $2\pi$ and area greater than $2\pi$, then some of the conical caps are on the wrong side of the planes supporting the triangles and those planes separate the points $v_1\dots v_k$. However, let $D_i$ be a disk with area greater than $2\pi$. Let $\Pi$ be any hyperplane passing through $v_0$ and the origin of $\spacetime$ that does not contain $v_i$. Reflecting $v_i$ through $v_0$ to obtain $v_i' = v_0 - (v_i - v_0)$ gives a point on the opposite side of $\Pi$. Now consider some triangle of the fan $v_0 v_i v_{i+1}$ and two other vertices $v_j$ and $v_k$ of the fan. Let $\Pi_i$ denote the hyperplane supporting the rays $D_0^*$, $D_i^*$, and $D_{i+1}^*$. If $v_j$ and $v_k$ correspond to disks $D_j$ and $D_k$ that have area strictly less than $2\pi$, then $v_j$ and $v_k$ lie on the same side of $\Pi_i$ as $D_j^*$ and $D_k^*$. However, if one has area greater than $2\pi$, say $D_i$, then $v_i$ lies on the wrong side. In this case replace $v_i$ with $v_i'$ its reflection through $v_0$. Let $v_1', \dots, v_k'$ denote the resulting points (which are reflected whenever a point is ``wrong'' and are left alone otherwise). We claim that the resulting triangle fan $v_0 v_1' v_2', v_0 v_2' v_3', \dots v_0 v_k' v_1'$ forms a triangle fan on a convex polyhedral cone. Indeed, by construction the plane $\Pi_i$ in $\mathbb{E}^3$ supporting a face $v_0 v_i' v_{i+1}'$ not only contains $v_0$, $v_i'$ and $v_{i+1}'$, but also contains $v_i$ and $v_{i+1}$. Furthermore, the corresponding hyperplane containing $\Pi_i$ and the origin of $\spacetime$ contains $D_0^*$, $D_i^*$, and $D_{i+1}^*$. By convexity the normal rays for all other disks in ${\bf D}$ lie on the same side of $\Pi_i$ since $0i(i+1)\in F(P)$. But each $v_i'$ lies on the same side of $\Pi_i$ as its corresponding disk $D_i$. Thus no two vertices of the fan are separated by $\Pi_i$. Therefore the fan is a convex triangle fan on a convex polyhedral cone. We call this convex triangle fan the \term{convex Euclidean triangle fan for vertex} $0$.

\subsection{Properties circle polyhedra}

We now derive some properties of circle polyhedra we use later. 

\begin{Lemma}\label{lem:noantipodaledges}
	Let $(P, {\bf D})$ be a strictly shallow circle polyhedron and $ij\in E(P)$. Then the centers of $D_i$ and $D_j$ are not antipodal.
\end{Lemma}
\begin{proof}
	Antipodal disks either have the same boundary circle or disjoint boundary circles. Since $(P, {\bf D})$ is shallow the boundary circles of neighboring disks must either be tangent or intersect. Thus the only possible antipodal case is when two neighboring disks share the same boundary circle.
	Assume for contradiction that there is an edge $ij\in E(P)$ such that $\partial D_i = \partial D_j$. If $D_i = D_j$ then $d(D_i, D_j) = -1$ contradicting that $(P, {\bf D})$ is strictly shallow. Suppose then that $D_i = -D_j$. Let $ijk\in F(P)$ be a face containing the edge $ij$. We may, by a Möbius transformation of $\mathbb{S}^2$, take $D_i$ and $D_j$ to the two hemispheres whose boundary is the equator ($D_i = (0, 0, 0, 1)$ and $D_j = (0, 0, 0, -1)$). Consider the disk $D_k = (a_k, b_k, c_k, d_k)$. The inversive distances from $D_k$ to $D_i$ and $D_j$ are given by $d(D_i, D_k) = -d_k$ and $d(D_j, D_k) = d_k$. Thus if $d_k\neq 0$ one of these represents a non-shallow overlap, a contradiction. Thus $d_k = 0$ and $D_k$ meets both disks $D_i$ and $D_j$ orthogonally. Now consider the other triangle $jil\in F(P)$ incident the edge $ij$. By the same argument $D_l$ must overlap both $D_i$ and $D_j$ orthogonally. But then there is a 4-cycle of edges $kj$, $jl$, $li$, $ik$ that all have inversive distance 0, contradicting that $(P, {\bf D})$ is strictly shallow.
\end{proof}

Because no two neighboring disks are antipodal, the shortest path between their centers is well defined.

\begin{Corollary}\label{cor:welldefinedcenter}
	Let $(P, {\bf D})$ be a strictly shallow circle polyhedron and $ij \in E(P)$. There is a unique shortest geodesic path on $\mathbb{S}^2$ between the centers of $D_i$ and $D_j$. 
\end{Corollary}

%Now, let $D_i$ and $D_j$ be two overlapping disks on $\mathbb{S}^2$ that do not share the same boundary. The intersection of the ray from the origin of $\widehat{\mathbb{E}}^3$ with the conical caps $D_i^*$ and $D_j^*$ with $\mathbb{S}^2$ is there center, and the projection of the line segment between $D_i^* D_j^*$ is the shortest path geodesic between the centers of $D_i$ and $D_j$. This remains true even if one of the disks, say $D_i$, has radius $2\pi$. In this case $D_i^*$ is the point at infinity in the same direction as the center of $D_i$ and $D_i^*D_j^*$ is a ray originating from $D_j^*$ in direction $D_i^*$. 

\begin{Lemma}\label{lem:geodesictriangles}
Let $(D_1, D_2, D_3)$ be a triple of disks with shallow overlaps. Let $s$ be the geodesic shortest path connecting the centers of $D_1$ and $D_2$. Then either 
\begin{enumerate}
	\item $D_3$ is disjoint from $s$,
	\item the boundary of $D_3$ is the great circle containing $s$, or
	\item $D_1$ and $D_2$ are tangent at a point $p$ on $s$ and $D_3$ is orthogonal to both $D_1$ and $D_2$ and tangent to the great circle containing $s$ at $p$. 
\end{enumerate}
\end{Lemma}
\begin{proof}
Assume that $D_3$ is not disjoint from $s$ and $\partial D_3$ is not the great circle containing $s$. We prove that $D_3$ is tangent to $s$ at a point $p$ that is also the point of tangency between $D_1$ and $D_2$. First, let us make several elementary remarks about disks and inversive distances. 

{\em Remark 1:} Let $D$ and $D'$ be disks such that $d(D, D') \in [-1, 1]$ and let $p$ be one of the intersection points (or sole point of tangency) between the boundary circles $\partial D$ and $\partial D'$. The spherical radius arc from the center of $D$ to $p$ contains a point of the interior of $D'$ if and only if the overlap between $D$ and $D'$ is greater than $\pi/2$, or equivalently if $d(D, D') < 0$. 

{\em Remark 2:} From remark 1 it follows immediately that if $d(D, D') \geq 0$ the center of $D$ is not contained within $D'$ and vice versa. 

{\em Remark 3:} If $D$ and $D'$ are two disks that do not contain each other's centers and $D''\subset D'$ is a disk with the same center as $D'$ but a strictly smaller radius, then $d(D, D') < d(D, D'')$. 

By remarks 1 and 2, $D_3$ cannot contain the endpoints of $s$ on its interior. Thus it is either tangent to $s$ at a point $p$ on its interior, or contains an interior segment $s'$ of $s$. 

In the first case, we show that $D_1$ and $D_2$ are tangent to each other at $p$. Indeed, assume not. Then $p$ is on the interior of either $D_1$ or $D_2$, say $D_1$. Let $r$ be the radius arc of $D_1$ along the segment $s$. By construction, $D_3$ is tangent to $r$ at $p$, which is interior to $D_1$. But then the radius arc between the center of $D_1$ and the intersection points of $\partial D_1$ with $\partial D_3$ is readily shown to intersect the interior of $D_3$. Then by remark 1, $D_1$ and $D_3$ do not have a shallow overlap, a contradiction. Thus $D_1$ and $D_2$ are tangent at $p$. The great circle containing $s$ meets both $D_1$ and $D_2$ orthogonally and thus, since $D_3$ is tangent to this great circle, it must be in the orthogonal coaxial family $(D_1\vee D_2)^\perp$.

Now assume that $D_3$ contains a proper sub-segment $s'$ of $s$. In this case, by remark 3, we can shrink $D_3$ while keeping its radius fixed until it is tangent to $s$ at a point $p$ to obtain a disk $D_3'$. By remark 3 $d(D_1, D_3') > d(D_1, D_3)$. But then by the paragraph above, $D_1$ and $D_2$ must be tangent at $p$ and $D_3'$ must meet them both orthogonally, so $d(D_1, D_3') = 0$. Therefore $d(D_1, D_3) < 0$, contradicting that the overlap is shallow. 
\end{proof}

Given a geodesic circle polyhedron $(P, \mathbf{D})$, let $\Delta_{ijk}$ denote the geodesic triangle corresponding to each face $ijk\in F(P)$. Let $F(i)$ denote the faces incident $i$. The \term{link} of $i$ is the union of the geodesic triangles incident $i$, $\textrm{lnk}(i) = \cup_{ijk\in F(i)} \Delta_{ijk}$. The link of vertex $i$ always contains its disk $D_i$ in a shallow geodesic circle polyhedron. 
\begin{Lemma}\label{lem:linkcontainment}
	Let $(P, {\bf D})$ be a strictly shallow geodesic circle polyhedron. Then for each vertex $i\in V(P)$ $D_i\cap\textrm{lnk}(i) = D_i$.  
\end{Lemma}
\begin{proof}
	Consider one of the geodesic triangles $\Delta_{ijk}$ making up the link of $i$. Let $s$ denote the geodesic arc on the boundary of $\Delta_{ijk}$ corresponding to $jk\in E(P)$. By lemma~\ref{lem:geodesictriangles} $D_i$ does not cross $s$ or contain it on its interior (it may contain it entirely, if $D_i$ is a great disk, be tangent to it, or be disjoint from it). On the other hand, $D_i$ does contain portions of the arcs of $\Delta_{ijk}$ corresponding to the other two edges $ik$ and $jk$ on its interior. Thus the boundary of $\textrm{lnk}(i)$ may be tangent to $D_i$ but does not cross into its interior. Furthermore, since $(P, \mathbf{D})$ is geodesic, the boundary of $\textrm{lnk}(i)$ is a simple spherical polygon. Thus $D_i$ is contained within $\textrm{lnk}(i)$. 
\end{proof}

Given that the link of any vertex contains the disk for a proper shallow circle polyhedron, we can now put bounds on the pairwise inversive distances between any two disks that are not neighbors in $P$.

\begin{Lemma}\label{lem:localtoglobal}
	Let $(P, \mathbf{D})$ be a shallow geodesic circle polyhedron. Then for all non-edge pairs $i, j \in V(P)$ $d(D_i, D_j) \geq 1$ with equality if and only if there exists vertices $l, k \in V(P)$ such that $ilk\in F(P)$, $klj \in F(P)$, and $d(D_i, D_l) = d(D_l, D_j) = d(D_j, D_k) = d(D_k, D_i) = 0$ and $d(D_k, D_l) = 1$ (and therefore $(P, {\bf D})$ is not strictly shallow).  
\end{Lemma}
\begin{proof}
	Let $i$ and $j$ be two distinct vertices of $P$ that are not connected by an edge. Let $T$ be the induced geodesic triangulation. Let $L_i$ and $L_j$ denote links of $i$ and $j$ in $T$. Since $T$ is a triangulation and $ij\not\in E(P)$, $L_i$ and $L_j$ are either disjoint, meet at a single vertex, or share a boundary edge between two vertices, say $l$ and $k$ where $ilk$ denotes the face of $L_i$ and $klj$ denotes the face of $L_j$ incident along the edge $lk$. By lemma~\ref{lem:linkcontainment}, $D_i \subset L_i$ and $D_j \subset L_j$. Thus if $L_i$ and $L_j$ are disjoint or meet at a vertex then $D_i$ and $D_j$ are disjoint. If $L_i$ and $L_j$ meet along an edge $kl$ then it is possible that $D_i$ and $D_j$ be tangent. This occurs only if $D_i$ is tangent to the boundary of $L_i$ and $D_j$ is tangent to the boundary of $L_j$. But by lemma~\ref{lem:geodesictriangles} this occurs if and only if $d(D_i, D_l) = d(D_l, D_j) = d(D_j, D_k) = d(D_k, D_i) = 0$ and $d(D_k, D_l) = 1$.
\end{proof}	

\section{Motions of disk patterns}\label{sec:motions}

Recall that a motion $(P, {\bf D}(t))$ is a path through the configuration space $\mathbb{R}^{4n}$ (see sec.~\ref{sec:abstractcpolyhedra}). In this section we investigate the properties of motions that maintain a set of soft constraints defined on the edges of a polyhedron $P$. Let an abstract triangulated polyhedron $P$ be given and let $\mathbf{D}$ be a realization of the vertex set $V(P)$ as a set of disks on $\mathbb{S}^2$. An unconstrained motion of $\mathbf{D}$ defined on an interval $I$ is a continuous family of realizations $\mathbf{D}(t)$ defined for $t\in I$ such that $\mathbf{D}(0) = \mathbf{D}$. Each disk $D_i \in \mathbf{D}$ has a continuous motion in $\mathbf{D}(t)$ given by a continuous path $D_i(t)$ in $\spacetime$. 

If, for each $ij\in E(P)$ and all $t\in I$ the inversive distance across the edge between $D_i$ and $D_j$ stays in the interval $[0, 1]$, we say that the motion is \term{shallow-constrained}. If further, at no point in the motion does a 4-cycle of edges all have inversive distance 0, we say that the motion is \term{strictly shallow-constrained}. Note that the inversive distances along each edge are constrained to the interval but are free to change within that interval. 

\subsection{Shallow-constrained motions of geodesic disk patterns}\label{sec:geodesicpreservation}

We now show that strictly shallow-constrained motions of disk patterns that start as geodesic circle polyhedra, remain geodesic circle polyhedra throughout the motion.

\begin{Theorem}\label{thm:geodesicpreservation}
	Let $(P, \mathbf{D}(t))$ be a strictly shallow-constrained motion of a disk pattern defined for $t\in I$. If there exists a $t'\in I$ such that $(P, \mathbf{D}(t'))$ is a geodesic circle polyhedron, then $(P, \mathbf{D}(t))$ is a geodesic circle polyhedron for all $t\in I$. 
\end{Theorem}
\begin{proof}
	For each $t\in I$, let $G(t)$ be the induced geodesic graph for $(P, \mathbf{D}(t))$ and let $n$ denote the number of vertices. Let $\mathbf{v}_t = [x_1(t)\ y_1(t)\ z_1(t)\ \dots\ x_n(t)\ y_n(t)\ z_n(t)]_t^T$ be the vector of coordinates of the vertex positions $p_i = (x_i, y_i, z_i)$ of the center $p_i$ of each disk $D_i$ on $\mathbb{S}^2$ in $\mathbb{E}^3$. The space of all possible $\mathbf{v}_t$ is a compact connected subset of $\mathbb{R}^{3n}$ which is called the configuration space of $G$. Let $\mathcal{C}$ denote the configuration space of $G$.  
	
	Since $\mathbf{D}(t)$ is continuous, $G(t)$ is a continuous family of graph drawings and ${\bf v}_t$ traces out a path in $\mathcal{C}$. Notice that if $G(t)$ is a triangulation, then there is an open neighborhood $U$ of ${\bf v}_t$ where all ${\bf u}\in U$ correspond to triangulations of $\mathbb{S}^2$. Similarly, if $G(t)$ has self-intersecting edges that intersect strictly on their interior, then there is a neighborhood of $G(t)$ in $\mathcal{C}$ in which all nearby drawings are strictly self-intersecting. Additionally, if each edge length is non-zero, then there is an open neighborhood of $G(t)$ in $\mathcal{C}$ in which all edge lengths are non-zero. 
	
	Therefore, if there exists at any time $t$ a graph $G(t)$ that is not a triangulation of $\mathbb{S}^2$ there must be a time $0 < t' \leq t$ at which $G(t')$ either has an edge $p_i p_j$ shrink to zero length, or a vertex $p_i$ intersect an edge $p_j p_k$ for some triangle $ijk \in F(G)$. In the first case, the inversive distance $d(D_i, D_j)$ is either $-1$, if $D_i = D_j$, or is less than $-1$, if $D_i \subset D_j$ or $D_j \subset D_i$. This contradicts that the motion is shallow-constrained. On the other hand, if $p_i$ intersects $p_j p_k$, then the disk $D_i$ intersects the geodesic shortest path segment between the centers of $D_j$ and $D_k$. But this contradicts lemma~\ref{lem:geodesictriangles} (note that $D_i$ cannot be the great disk whose boundary $\partial D_i$ contains $p_j p_k$ and have its center on $p_j p_k$ so condition 1 of lemma~\ref{lem:geodesictriangles} is not applicable).
\end{proof}

%TODO CONTINUE FROM HERE
\subsection{Shallow-constrained motions of geodesic convex disk patterns}\label{sec:convexpreservation}

Similar to the previous section, we show that strictly shallow-constrained motions of disk patterns that start out as both geodesic and strictly convex circle polyhedra stay strictly convex circle polyhedra throughout the motion.

\begin{Theorem}\label{thm:convexpreservation}
	Let $(P, \mathbf{D}(t))$ be a  strictly shallow-constrained motion of a disk pattern defined for $t\in [0, T)$ for some $T\in\mathbb{R}\cup\{\infty\}$. Let $n > 4$ be the number of vertices in $P$. If $(P, \mathbf{D}(0))$ is a strictly convex geodesic circle polyhedron, then $(P, \mathbf{D}(t))$ is a strictly convex circle polyhedron for all $t\in [0, T)$. 
\end{Theorem}
\begin{proof}	
We first note that for every $ij\in E(P)$ $D_i(t)\neq D_j(t)$ at any time $t\in [0, T)$. To the contrary, if $D_i(t)=D_j(t)$ then $d(D_i(t), D_j(t)) = -1$ and the motion is not strictly shallow. 

Now assume for contradiction that at some time $t'\in [0, T)$ $(P, {\bf D}(t'))$ is not strictly convex. By lemma~\ref{lem:maintainconvexity1} there must be a time $t''\in[0, T)$ where $(P, {\bf D}(t''))$ has two neighboring faces $ijk$ and $kjl$ such that $D_i(t'')$, $D_j(t'')$, $D_k(t'')$, and $D_l(t'')$ are coplanar. 

Let $\Pi$ denote the c-plane containing $D_i(t'')$, $D_j(t'')$, $D_k(t'')$, and $D_l(t'')$. By lemma~\ref{lem:cplanarity}, $\Pi$ is not hyperbolic. If it is parabolic, then the four cycle of edges $ij$, $jl$, $lk$, $ki$ all have inversive distance 0, contradicting the strictly shallow assumption of the hypothesis. If it is elliptic, then the union $D_i(t'') \cup D_j(t'') \cup D_k(t'') \cup D_l(t'')$ covers the sphere $\mathbb{S}^2$ and any other disk $D_m(t'')$ has a negative inversive distance to at least one of $D_i(t'')$, $D_j(t'')$, $D_k(t'')$, or $D_l(t'')$. Without loss of generality assume that $d(D_i(t''), D_m(t'')) < 0$. If $im\in E(P)$ this contradicts the shallowness assumption. Suppose $im\not\in E(P)$. By theorem~\ref{thm:geodesicpreservation} $(P, \mathbf{D}(t''))$ is geodesic. But then by lemma~\ref{lem:localtoglobal}, $d(D_i(t''), D_m(t'')) > 1$, a contradiction. 
\end{proof}

\subsection{Radius collapsing and encompassing motions of shallow circle packings are avoidable}

In this section we show that the extra conditions in the statement of the Koebe-Andre'ev-Thurston theorem (conditions 1 and 2) allow us to avoid motions of disk patterns in which a disk approaches a point-disk (either by its area approaching 0 or $4\pi$). One of our main analysis tools comes from the following observation. Suppose we reparametrize each disk by the coordinates of its center $(x, y, z)$ on the sphere $\mathbb{S}^2$ and its spherical radius $\rho\in[0, \pi]$, with $\rho = 0, \pi$ denoting a disk whose boundary has degenerated into a single point. Then the configuration of one disk is a subset of $\mathbb{R}^4$ which lies within the region defined by the bounds $-1 \leq x, y, z \leq 1$ and $0 \leq \rho \leq \pi$. A parametrization of $n$ disks similarly falls within a bounded region of $\mathbb{R}^{4n}$. Consider any motion ${\bf D}(t)$ of disks under this parametrization defined for $t\in[0,T)$ for some $T\in\mathbb{R}^+\cup\{\infty\}$. The motion traces a path in the configuration space $\mathbb{R}^{4n}$, which by the discussion above falls within a bounded region. Then by the Bolzano-Weierstrass theorem, there exists a convergent subsequence. We now ask, when is it possible that as $t\rightarrow T$, the radius of some disk $D_i(t)$ approaches zero?

In general, we can by a Möbius flow collect all disk boundary circles down to a point. Simply select two antipodal points neither of which is contained on the boundary of any disk and then compute the standard Möbius flow out from one antipodal point towards the other. All points of the sphere collect at the sink pole and thus every disk approaches either 0 or $4\pi$ area. To avoid this situation we must additionally pin-down three disks. Let $ijk \in F(P)$. We call a motion $(P, {\bf D}(t))$ a \term{pinned motion} with respect to $ijk$ if $D_i(t)$, $D_j(t)$, and $D_k(t)$ are constant throughout the motion. %A motion in which a disk's area approaches 0 we call \term{radius collapsing} and one in which a disk's area approaches $4\pi$ we call \term{encompassing}. Our goal now is to derive properties that guarantee a pinned motion is neither radius collapsing nor encompassing.

\begin{Lemma}\label{lem:contractionimplications}
	Let $(P, {\bf D}(t))$ be a strictly shallow-constrained motion of a circle polyhedron for $t\in [0, T)$ for some $T\in\mathbb{R}^+\cup\{\infty\}$ with $(P, {\bf D}(t))$ geodesic. Suppose some disk boundary $\partial D_i$ vanishes to the a point as $t\rightarrow T$ (but at no earlier time). Let $V$ denote the maximal set of vertices with vanishing boundary circles that are edge-connected to $i$  and $V'$ denote the set of vertices with boundaries that do not contract but are connected by an edge to a vertex in $V$. Then $|V'|\leq 4$. 
	
	Furthermore:
	\begin{enumerate}
		\item If $|V'| = 4$, then the vertices form a 4-cycle of edges that limit to inversive distance 0 along each edge and the edges do not bound two neighboring faces of $P$. 
		\item If $|V'| = 3$, then the vertices form a 3-cycle of edges that limit to inversive distance 0 along each edge and it is not possible that all three are pinned. Furthermore, the sum of the overlap angles of the 3-cycle converges to $\pi$. 
		\item If $|V'| \leq 2$, then $V \cup V' = V(P)$ and it is not possible that three disks are pinned in $(P, {\bf D}(t))$.
	\end{enumerate}% TODO
\end{Lemma}
\begin{proof}
	Suppose the boundary of some disk $D(t)$ vanishes as $t\rightarrow T$. By the Bolzano-Weierstrass theorem there is a convergent subsequence of $(P, {\bf D}(t))$. Let $D_i'$ denote the disk that $D_i(t)$ converges to in this subsequence. Call this the convergent set ${\bf D}$. By theorem~\ref{thm:geodesicpreservation}, $(P, {\bf D}(t))$ is geodesic for all $t\in[0, T)$. 
	
	{\em Inversive Distance Convergence Property (IDCP).} By strict shallowness, for every edge $ij\in E(P)$, $d(D_i(t), D_j(t)) \in (0, 1]$ and thus $d(D_i', D_j') \in [0, 1]$. Because its geodesic, for a non-edge pair $i, j\in V(P)$, $d(D_i(t), D_j(t)) > 1$ (by lemma~\ref{lem:localtoglobal}), and thus $d(D_i', D_j') \geq 1$. 	
	Let $V$ be the maximal edge-connected set (to the vanishing disk $D(t)$) of vertices whose boundary circles vanish and $V'$ be the vertices connected by an edge in $P$ to a vertex in $V$ with boundary circles that do not vanish. Every disk in $V$ converges to a single point $p$ in the convergent set. 
	
	{\em Distinct disk property.} We first claim that it is not the case that $D_i' = D_j'$ for any distinct $i, j\in V'$. Assume not.  The radii of $D_i'$ and $D_j'$ are bounded away from $0$. Then the inversive distance is $d(D_i', D_j') = -1$ contradicting (IDCP).
	
	\paragraph{Proof that $|V'| < 5.$}
	We next claim that $|V'| < 5$. For contradiction assume that $V'$ contains at least five vertices. Let $D_1'$, $D_2'$, $D_3'$, $D_4'$, and $D_5'$ denote the disks corresponding to these vertices. By the discussion above these five disks are distinct. Since each of them is the convergence disk of a disk overlapping some disk in $V$ and everything in $V$ converges to a point $p$, the boundary circles $\partial D_1', \dots, \partial D_5'$ all contain $p$. Now, by theorem~\ref{thm:geodesicpreservation}, $(P, {\bf D}(t))$ is geodesic for all $t\in[0, T)$. Then by corollary~\ref{cor:no5shallowinparabolic} at least two of the disks must have a non-shallow overlap, which contradicts (IDCP).
	
	\paragraph{Analysis of $|V'|=4$.}
	If $|V'|=4$, the four disks $D_1'$, $D_2'$, $D_3'$, and $D_4'$ meet at a point. By lemma~\ref{lem:paraboliccases}, in order to not contradict (IDCP), two pairs of these disks are tangent and the pairs are mutually orthogonal. Without loss of generality assume that $(D_1', D_3')$ are tangent and $(D_2', D_4')$ are tangent, and the two pairs are mutually orthogonal. 
	
	Then $D_1'$, $D_2'$, $D_3'$, $D_4'$ is a cycle of disks each of which is (cyclically) orthogonal to the next. Thus at all $t\in[0, T)$ at least one of the overlap angles for the edges $12$, $13$, $24$, and $34$ is greater than $0$ but in the limit equal $0$. Assume that the cycle of edges $(12, 23, 34, 41)$ bounds two adjacent faces of $P$. Then all of the other vertices must be in $V$ and all of the disks save $D_1$ through $D_4$ limit to the point common to $D_1'$, $D_2'$, $D_3'$, $D_4'$. But then the geodesic triangulation induced by $(P, {\bf D}(t))$ approaches the configuration of geodesic triangles $p_1 p_2 p$, $p_2 p_3 p$, $p_3 p_4 p$, $p_4 p_1 p$ where $p_i$ denotes the center of $D_i'$ and $p$ is the point corresponding to the collapsed disks in $V$. But the outer boundary of this configuration is a 4-cycle $(p_1, p_2, p_3, p_4)$ and thus not a triangulation of $\mathbb{S}^2$. But then for sufficiently large $t$, the geodesic graph induced by $(P, {\bf D}(t))$ is not a triangulation. This contradicts theorem~\ref{thm:geodesicpreservation}.
		
	\paragraph{Analysis of $|V'|=3$.} In this case the three disks $D_1'$, $D_2'$, $D_3'$ must all have pairwise inversive distances in $[0, 1]$ and thus $123$ is a cycle of edges in $P$. By lemma~\ref{lem:paraboliccases}, the sum of the overlap angles is $\pi$. Now suppose all three disks are pinned: $D_1(t) = D_1'$, $D_2(t) = D_2'$, and $D_3(t) = D_3'$ are constant. By strict shallowness, none of the three are tangent at $p$, and thus any small disk containing $p$ will overlap at least one of the disks by more than $\pi/2$. Thus any of the disks corresponding to vertices in $P$ will overlap at least one of them by more than $\pi/2$. But any of the disks in $V$ converge to small disks containing $p$, and thus at some time $t'\in[0, T)$ there is a disk $D_i$ for some $i\in V$ overlapping one of the disks, say $D_1(t')$ such that $d(D_i(t'), D_1(t')) < 0$ contradicting the shallow overlap.
	
	\paragraph{Analysis of $|V'| \leq 2$.} In this case, since $V'$ is the boundary of a connected set of vertices $V$ on a triangulated polyhedron and $|V'|\leq 2$ it must be the case that every vertex of $P$ is either in $V$ or in $V'$. Then if any three disks are pinned at least one of the disks in $V$ is pinned and thus it cannot vanish, a contradiction.
\end{proof}

We now show that the extra conditions from the Koebe-Andre'ev-Thurston theorem (in the limit), coupled with three disks whose radii are bounded away from zero guarantees that no disk vanishes for any strictly shallow motion of a disk pattern that starts out geodesic and has three pinned disks.

\begin{Theorem}\label{thm:katconditionsnocollapse}
	Let $(P, {\bf D}(t))$ be a strictly shallow motion of a disk pattern for $t\in[0,T)$ for some $T\in\mathbb{R}^+\cup\{\infty\}$ such that $(P, {\bf D}(0))$ is geodesic. Suppose further that for one face $ijk$ the disks $D_i(t)$, $D_j(t)$, and $D_k(t)$ are constant. Finally, letting $d(e(t))$ denote the inversive distance between the disks corresponding to the endpoints of an edge $e\in E(P)$ at time $t\in[0, T)$, suppose that: 
	\begin{enumerate}
		\item If $e_1$, $e_2$, $e_3$ form a closed loop of edges such that at some $t_0\in[0, T]$ $$\lim_{t\rightarrow t_0} \sum_{k=1}^3 \arccos d(e_k(t)) \geq \pi,$$ then $e_1$, $e_2$, and $e_3$ bound a face of $P$. 
		\item If $e_1$, $e_2$, $e_3$, and $e_4$ form a closed loop of edges such that at some $t_0\in [0, T]$, $$\lim_{t\rightarrow T}\sum_{i=1}^4 \arccos d(e_i) = 2\pi$$ then $e_1$, $e_2$, $e_3$, and $e_4$ bound the union of two faces of $P$. 
	\end{enumerate}
	Then no disk in ${\bf D}(t)$ vanishes as $t\rightarrow T$. 
\end{Theorem}

\begin{proof}
	Since $(P, {\bf D}(0))$ is geodesic, by theorem~\ref{thm:geodesicpreservation}, $(P, {\bf D}(t))$ is geodesic for all $t\in[0,T)$. 
	
	Suppose now that some disk's radius vanishes and let $V$ be a maximally edge-connected set of disks whose radius vanishes. Let $V'$ denote the vertices of $V(P)\backslash V$ connected by an edge in $E(P)$ to a vertex in $V$. By lemma~\ref{lem:contractionimplications}, $|V'| \leq 4$. 
	
	If $|V'| \leq 2$, then it is not possible that three disks are pinned, a contradiction.
	
	Suppose that $|V'| = 3$. Then by lemma~\ref{lem:contractionimplications} the vertices of $V'$ are connected by a three cycle of edges $e_1$, $e_2$, and $e_3$ whose angle sum approaches $\pi$. Then by hypothesis $e_1$, $e_2$, and $e_3$ bound a face of $P$. Then every vertex not in $V'$ must be in $V$. But none of the vertices in $V$ can be constant, since their boundary circles contract to a point. But then the vertices in $V'$ must be pinned contradiction lemma~\ref{lem:contractionimplications}.

	Finally, suppose that $|V'| = 4$. Then by lemma~\ref{lem:contractionimplications} there is a 4-cycle of edges $e_1$, $e_2$, $e_3$, and $e_4$ connecting the vertices of $E(P)$ whose inversive distances all approach $0$ and this 4-cycle does not bound two neighboring faces of $P$.  But by hypothesis, these edges must bound two neighboring faces, a contradiction. 
\end{proof}

\section{Infinitesimal rigidity of convex circle packings}\label{sec:infrigidity}

We now develop the infinitesimal rigidity of strictly convex circle packings. This section generalizes the results of \cite{BBP19} in the context of convex circle polyhedra. 

\subsection{Rigidity matrix}

Let $(P, \mathbf{D})$ be a triangulated circle polyhedron with $n$ vertices. Since $P$ is triangulated it has $m = 3n-6$ edges. Think of $\mathbf{D}$ as  a configuration in the configuration space $\mathbb{R}^{4n}$: 
\[
\mathbf{D} = [ a_1\ b_1\ c_1\ d_1\ \dots \ a_n\ b_n\ c_n\ d_n]^T\in \mathbb{R}^{4n}.
\]
Where the coordinates of each disk $D_i = (a_i, b_i, c_i, d_i)$ are normalized to lie on the de Sitter sphere. Then $d(D_i, D_j)=-\langle D_i, D_j \rangle_{1,3}$. Define a measurement function $f:\mathbb{R}^{4n}\rightarrow\mathbb{R}^{4n-6}$, which measures the negative length of each edge $ij\in E(P)$ as well as half the squared Minkowski norm of each vertex. Thus $f$ has $3n-6$ measurements $f_{ij}$ indexed by the edges of $P$ and $n$ measurements $f_i$ indexed by the vertices of $P$. The entries of $f$ are defined by 
\begin{equation}
	f_{ij}(\mathbf{D})=-d(D_i, D_j) = \langle D_i, D_j\rangle_{1,3}
\end{equation}
and
\begin{equation}
	f_i(\mathbf{D})=(1/2)\langle D_i, D_i\rangle_{1,3}.
\end{equation}
Let $J$ denote the Jacobian matrix of $f$. $J$ is called the \term{rigidity matrix} of $(P, {\bf D})$. Each row corresponds to either an edge $ij$ or a vertex $i$ of $P$. The columns represent the $4n$ coordinates of the configuration $\mathbf{D}$. The row corresponding to the $ij$ edge of $J$ is zero everywhere except at the entries corresponding to the coordinates of $D_i = (a_i, b_i, c_i, d_i)$ and $D_j = (a_j, b_j, c_j, d_j)$. The entries of the row are
\begin{equation}
\begin{array}{cccccccccccc}
	 & \dots & a_i & b_i & c_i & d_i & \dots & a_j & b_j & c_j & d_j & \dots \\ 
	J_{ij}=& (\dots 0 \dots & a_j & -b_j & -c_j & -d_j & \dots 0 \dots & a_i & -b_i & -c_i & -d_i & \dots 0 \dots ).
\end{array}
\end{equation}
(Notationally, in the equation above, the first line gives the column labels and the second line the definition of the row $J_{ij}$.)
The row corresponding to the vertex $i$ of $J$ is similarly zero everywhere except at the entries corresponding to the coordinates of $D_i$:
\begin{equation}
\begin{array}{ccccccc}
	 & \dots & a_i & b_i & c_i & d_i & \dots  \\ 
	J_{i}=& (\dots 0 \dots & a_i & - b_i & - c_i & - d_i & \dots 0 \dots  ).
\end{array}
\end{equation}
By construction, the null space of $J$ corresponds to the space of infinitesimal motions of the disk set $\mathbf{D}$ that maintain the inversive distances on each edge $ij\in E(P)$ and maintain that each vertex $i\in V(P)$ remains on the de Sitter sphere. Since the Lorentz transformations map $\mathrm{d}\mathbb{S}^3$ to itself and maintain the Minkowski inner product, they form a 6-dimensional family of trivial motions. Thus the dimension of the null space of $J$ must be at least 6, corresponding to the ininitessimal Lorentz transformations. By the rank-nullity theorem $J$ has rank at most $4n-6$. If the rank of $J$ equals $4n-6$, then the only motions of $(P, {\bf D})$ must be Lorentz transformations, and therefore trivial. In this case we say that $(P, {\bf D})$ is \term{infinitesimally rigid}. Otherwise, we say that $(P, {\bf D})$ is infinitesimally flexible. 

\paragraph{Proof of infinfinitesimal rigidity of strictly shallow convex geodesic circle polyhedra.}\label{sec:cauchylemmas}

We now prove that strictly convex geodesic circle polyhedra are infinitesimally rigid. We first remind the reader of the Cauchy index lemma, which have at this point become a standard tool for analyzing polyhedra. We state Cauchy's index lemma in two forms: the weak Cauchy index lemma, which was used by Cauchy to prove the global rigidity of convex Euclidean polyhedra; and a stronger form, which implies the weak.

Let $G$ be a graph drawn on a topological sphere with arcs representing edges such that no two arcs intersect and no vertex touches an edge it is not incident to. Call this a \term{topologically planar graph drawing}. The orientation on the sphere imposes a cyclic ordering of the edges incident around each vertex. Label each edge of $G$ with a sign $+$, $-$, or 0. Around any vertex of $G$ in order we may list the signs in cyclic order (ignoring 0) and count the number of sign changes in the list. Call the number of sign changes around a vertex $v$ the \term{index of $v$}, denoted $\mathrm{ind}(v)$. Since the sign changes are a rotation around each vertex, the index of any vertex is even. Cauchy's weak index lemma follows. 
\begin{Lemma}[Weak Cauchy index lemma]\label{lem:weakcauchyindex}
Let $G$ be a topologically planar graph drawing. Label each edge of $G$ with $+$, $-$, or $0$ such that at least one edge of $G$ is not labeled 0. Let $\textrm{ind}(v)$ denote the index of each vertex $v \in V(G)$. Then there exists at least one vertex $v$ incident to an edge labeled $+$ or $-$ such that $\textrm{ind}(v)\in\{0, 2\}$. 
\end{Lemma}
The strong form of the lemma involves the sum of the indices and easily implies the weak form. 
\begin{Lemma}[Strong Cauchy index lemma]\label{lem:strongcauchyindex}
Let $G$ be a geodesic graph drawing with $n$ vertices. Label each edge of $G$ with $+$, $-$, or $0$ such that at least one edge of $G$ is not labeled 0. Let $\textrm{ind}(v)$ denote the index of each vertex $v \in V(G)$. Let $n'$ denote the number of vertices incident an edge labeled $+$ or $-$. Let $s = \sum_{v\in V(G)} \textrm{ind}(v)$. Then $s\leq 4n' - 8$. \end{Lemma}

We now establish the following lemma, which implies the infinitesimal rigidity of strictly convex circle polyhedra.
 
\begin{Lemma}\label{lem:emptycokernel}
Let $(P, {\bf D})$ be a strictly convex geodesic circle polyhedron with $P$ not a tetrahedron. Let $J$ be its rigidity matrix. Then the rank of $J^T$ is $4n-6$. 
\end{Lemma}
\begin{proof}
	By the rank-nullity theorem, $\textrm{rank}\, J^T = 4n-6$ if and only if $J^T{\bf v} = {\bf 0}$ has no non-trivial solutions. We prove this by contradiction. Assume $J^T{\bf v} = {\bf 0}$ and ${\bf v}\neq{\bf 0}$. The vector ${\bf v}$ has an entry for each edge $ij\in E(P)$, which we denote by $\omega_{ij}$, and an entry for each vertex $i\in V(P)$, which we denote by $\omega_i$. 
	
	Consider the four rows of $J^T$ corresponding to a vertex $i\in V(P)$. The condition that $J^T{\bf v}={\bf 0}$ restricted to the four coordinate rows of vertex $i$ in $J^T$ is equivalent to the vector equation
	\begin{equation}\label{eq:equilibriumcond}
		{\bf 0} = \omega_i (a_i, -b_i, -c_i, -d_i)^T + \sum_{ij\in E(P)} \omega_{ij}(a_j, -b_j, -c_j, -d_j)^T.
	\end{equation}
	We first remark that from this formulation we see immediately that it is not possible for $\omega_{ij} = 0$ for all $ij\in E(P)$ and $\omega_i \neq 0$ since the coordinates of a disk on the de Sitter sphere cannot all be 0. Obviously, if (\ref{eq:equilibriumcond}) is satisfied, then we may reflect through the spatial coordinates to obtain
	\begin{equation}\label{eq:equilibriumcond2}
		{\bf 0} = \omega_i (a_i, b_i, c_i, d_i)^T + \sum_{ij\in E(P)} \omega_{ij}(a_j, b_j, c_j, d_j)^T.
	\end{equation}
	
	%Recall that each vector $D_i$ may be normalized with the Euclidean norm to obtain a point $D_i^*$ in $\mathbb{S}^3$. Let $P^*$ denote the polyhedron in $\mathbb{S}^3$ corresponding to $(P, {\bf D})$. $P^*$ is convex if and only if $(P, {\bf D})$ is.
	
	Now, label each edge $ij \in E(P)$ with a $+$, a $-$, or a $0$ depending on whether $\omega_{ij} > 0$, $\omega_{ij}<0$, or $\omega_{ij}=0$. Since at least one edge corresponds to a non-zero entry $\omega_{ij}$ in ${\bf v}$, there must be an edge of labeled with either a $+$ or a $-$. Then by the weak Cauchy index lemma there is a vertex $0\in V(P)$ such that the counterclockwise order of labeled edges (those that are either $+$ or $-$) around $0$ has at most two sign changes. 
	
	Assume first that there are two sign changes. Let $1, \dots, k \in V(P)$ denote the labeled neighbors of $0$ in order indexed so that edge $01$ is the labeled with a $+$, $0k$ is labeled with a $-$, and $i$ is the largest index such that $0i$ is labeled with a $+$. 
	Without loss of generality assume that no disk has area equal to $2\pi$ since we may apply a nearby Möbius transformation to ensure this property, and by construction the Möbius transformations are the trivial motions and thus do not change the rank of our Jacobian.

	Denote the vertices of the cap polyhedron triangle fan for vertex 0 by $v_0, v_1, \dots, v_k$ and let $v_0, v_1', \dots, v_k'$ denote the vertices of the corresponding convex Euclidean triangle fan ((as defined in section~\ref{sec:convexcpolyhedradef}). Since $v_0, v_1', \dots, v_k'$ is convex, there is a Euclidean plane $\Pi$ through $v_0$ that separates $v_1', \dots, v_i'$ from $v_{i+1}', \dots, v_k'$ and contains none of them on its interior. 
	
	Let $\Pi'$ denote the hyperplane through the origin of $\spacetime$ containing $\Pi$. By construction, if $D_j$ has area less than $2\pi$, then $v_j$ lies on the same side of $\Pi$ as $v_j'$ and $v_j'$ lies on the same side of $\Pi'$ as the ray $D_j^*$. On the other hand if $D_j$ has area greater than $\pi$ then $v_j$ and $v_j'$ are on opposite sides of $\Pi$ and the ray $D_j^*$ and $v_j'$ lie on the same side of $\Pi'$ (see section~\ref{sec:convexcpolyhedradef}). Thus $\Pi'$ separates the rays corresponding to the disks with $+$ signs from those with $-$ signs. Now, scale each of $D_1^*\dots D_k^*$ by the appropriate value $\omega_{0j}$. We have that $D_{0j}^*$ lies on the same side of $\Pi'$ as $\omega_{0j} D_{0j}^*$ when $\omega_{0j} > 0$, and on opposite sides otherwise. Since $\Pi'$ separates the vectors with positive $\omega_{0j}$ from those with negative $\omega_{0j}$, this scaling moves all scaled vectors to the same side of $\Pi'$. But then all the scaled vectors lie on the same side of a hyperplane through the origin and therefore their sum cannot be ${\bf 0}$. Thus (\ref{eq:equilibriumcond2}) is not satisfiable with two sign changes.

	%Let $\Pi'$ denote the hyperplane through the origin of $\mathbb{R}^{1,3}$ whose intersection with $\mathbb{E}^3$ is $\Pi$ and let $v_0', \dots, v_k'$ be the conical cap vectors of the disks corresponding to the vertices $v_0, \dots, v_k$. Since $\Pi$ passes through $v_0$, $\Pi'$ contains the vector $v_0'$ and since $\Pi$ separates $v_0, \dots, v_i$ from $v_{i+1}, \dots, v_k$, $\Pi'$ separates the vectors $v_0', \dots, v_i'$ in $\mathbb{R}^{1,3}$ from the vectors $v_{i+1}', \dots, v_k'$. Now scale each $v_i'$ to obtain $\hat{v}_i$ by its appropriate entry in ${\bf v}$. $v_1', \dots, v_i'$ are scaled by a positive value $\hat{v}_1, \dots, \hat{v}_i$ lie on the same side of $\Pi'$ as $v_1', \dots, v_i'$. $v_{i+1}', \dots, v_k$ are each scaled with a negative value and thus $\hat{v}_{i+1}, \dots, \hat{v}_k$ lie on the opposite side of $\Pi'$ as $v_{i+1}', \dots, v_k'$. Therefore $\hat{v}_1, \dots, \hat{v}_k$ all lie in the same half-space whose boundary is $\Pi'$. Since none of the vectors $\hat{v}_1, \dots, \hat{v}_k$ are contained within $\Pi'$ (by construction) and all lie on the same side of $\Pi'$, the sum $\hat{v}=\sum_{i=1}^k \hat{v}_i$ is strictly contained in the half-space bounded by $\Pi'$ that contains the scaled vectors. But $\Pi'$ contains the origin of $\spacetime$, so $\hat{v}_0 + \hat{v} \neq {\bf 0}$. 
	
	Essentially the same argument holds when the labels all have the same sign. Here we select $\Pi$ to be a plane through $v_0$ such that all other $v_i'$ lie on the same side. Then all the values are either all scaled by a positive scalar or all scaled by a negative scalar. In either case the same argument shows that (\ref{eq:equilibriumcond2}) cannot be satisfied. The lemma follows. 
	\end{proof}

From this lemma a version of the Legendre-Cauchy-Dehn lemma for strictly convex circle polyhedra follows immediately.

\begin{Corollary}
	Strictly convex circle polyhedra are infinitesimally rigid.
\end{Corollary}

\paragraph{Remark.} The above result on the infinitesimal rigidity of strictly convex circle polyhedra may also be obtained by applying a theorem from the unpublished manuscript in \cite{SW07} to the connection between circle polyhedra and de Sitter polyhedra.

\subsection{The square rigidity matrix}

Let $(P, \mathbf{D})$ be a convex circle polyhedron with $n$ vertices such that every disk $D_i = (a_i, b_i, c_i, d_i)$ is given in de Sitter coordinates and let $J$ be its $(4n-6)\times 4n$ rigidity matrix. $J$ has rank $4n-6$ by corollary~\ref{cor:Jrank}. The nullspace of $J$ is 6-dimensional, corresponding to the dimension of the Lorentz group on $\spacetime$ (equivalently the Möbius group on $\mathbb{S}^2$). As in \cite{CG19}, we want to mod out the Möbius group by the coordinates of three disks. We do this by fixing one face $ijk\in F(P)$ and removing the $b_i$, $c_i$, $d_i$, $b_j$, and $c_j$ columns and one of the columns corresponding to the spacial coordinates of $D_k$. The resulting matrix $J_{ijk}$ is a square $(4n-6)\times (4n-6)$ matrix called the \term{square rigidity matrix} for the face $ijk\in F(P)$. We now show that this matrix has full rank.

\begin{Lemma}\label{lem:squarerigiditymatrixrank}
	The rank of the square rigidity matrix $J_{ijk}$ is $4n-6$. 
\end{Lemma}

\begin{proof}
	We begin by pinning $b_i$, $c_i$, $d_i$, $b_j$, and $c_j$. By lemma~\ref{lem:twosolutionsgivenbc} any continuous motion fixing these five coordinates must fix the two disks $D_i$ and $D_j$, and therefore all 8 coordinates of the two disks. In standard rigidity theory parlance, this is called \term{pinning}. Fixing two disks determines a unique Möbius flow of the third disk $D_k$. The flow must vary at least one of the space coordinates of $D_k$. Pinning this column stops the flow, since the derivative of the flow must be zero at this point and all other coordinates are determined by these 6 choices. Without loss of generality, assume this is the $b_k$ column. Now, add 6 rows to $J$, one for each of the pinned coordinates, which is 0 everywhere except with a 1 in the pinned coordinate column. We call the resulting $4n\times 4n$ matrix $\bar{J}$ the \term{pinned rigidity matrix}. 
	
	Assume there is a non-trivial $\mathbf{v}$ such that $\bar{J}\mathbf{v} = 0$. By the discussion above, $\mathbf{v}$ must have 0's in all of its rows corresponding to the coordinates of vertices $i$, $j$, and $k$ and is therefore not the derivative of a Möbius transformation. But by construction, $J\mathbf{v} = 0$ as well and thus $\mathbf{v}$ is in the kernel of $J$. But the kernel of $J$ is corresponds exactly to the derivatives of the Möbius transformations on $\mathbf{D}$, a contradiction. Therefore $\bar{J}$ has rank $4n$. Then every column of $\bar{J}$ is linearly independent of the others and in particular the columns corresponding to the $4n-6$ non-pinned coordinates must all be linearly independent of each other. But these were simply padded with 0 entries to obtain $\bar{J}$ and the corresponding columns must also be linearly independent in $J$. Thus if we delete the 6 columns corresponding to the pinned coordinates from $J$ we obtain a $4n-6\times 4n-6$ matrix of full rank. 
\end{proof}

\section{Producing overlap packings}\label{sec:proofofkat}

\subsection{The Connelly-Gortler flip-and-flow algorithm}

In \cite{CG19}, Connelly and Gortler describe an algorithm for computing a tangency packing of a given abstract triangulated polyhedron using a finite number of operations they call a flip-and-flow operation. Suppose $P$ and $P'$ are abstract triangulated polyhedra that combinatorially differ by a single edge flip operation and $(P, {\bf D})$ is a tangency packing of disks ${\bf D}$ in the plane. Let $P^-$ denote the common vertices, edges, and faces of $P$ and $P'$. $P^-$ has the same number of vertices, one fewer edge, and two fewer faces than $P$ and $P'$. The tangency packing $(P, {\bf D})$ is infinitesimally rigid. The removal of one edge gives rise to a single non-trivial infinitesimal motion and the configuration space of disks realizing tangencies along the edges of $P^-$ is a 1-dimensional manifold they call the packing manifold. Moving along this manifold gives rise to a continuous motion ${\bf D}(t)$ of the disks with ${\bf D}(0) = {\bf D}$ which maintains tangencies along every edge of $P^-$. Meanwhile as $t$ increases, the inversive distance between the disks corresponding to the endpoints of the edge $e \in E(P)\slash E(P^-)$ is initially equal to 1 and strictly increases while the inversive distance between the two disks corresponding to the endpoints of the edge $e'\in E(P')\slash E(P^-)$ is initially greater than one and strictly decreases. The motion continues until the inversive distance along $e'$ becomes $1$ at some finite time $t_1$. At that point $(P', {\bf D}(t_1))$ is a tangency packing. Thus any combinatorial flip can be realized by a continuous motion that keeps all neighboring circles tangent except across the edge that needs to flip. 

Any two triangulations with the same number of vertices are connected by a finite sequence of combinatorial flip operations. Thus, if there exists any canonical triangulation that can be shown to have a circle packing, then all triangulations with the same number of vertices are reachable via the Connelly-Gortler flip-and-flow operation. In fact, such canonical circle packings are easy to come by. Select any three mutually tangent disks in the plane. These disks define an interstice region and there is a unique disk on the interior of this interstice whose boundary is tangent to the other three. Thus inductively a circle packing with $n$ disks whose contact graph is an abstract triangulated polyhedron may be obtained by starting with any three mutually tangent disks and repeatedly filling in interstices inductively $n-3$ times. The resulting triangulation is the starting point for the Connelly-Gortler algorithm from which all possible triangulated tangency packings may be obtained. 

We take the end of the Connelly-Gortler algorithm as our starting point which produces a tangency circle packing in the plane for a desired contact graph. We take this as a disk packing using the Euclidean interior of each circle. Given such a tangency disk packing in the plane, a tangency packing on $\mathbb{S}^2$ may be obtained via stereographic projection onto the sphere such that no disk has area greater than $2\pi$. The conical cap polyhedron for $P$ is known as its \term{Koebe polyhedron} and is known to be convex, which by lemma~\ref{lem:convexcappolyhedron} implies $(P, {\bf D})$ is convex in our sense. Furthermore, the geodesic arc between to neighboring disks stays entirely within their union and passes through the point of tangency. Thus, because the disks are each disjoint, the induced geodesic graph is a geodesic triangulation. In our terminology, the tangency circle packing theorem is: 
\begin{Theorem}[Tangency Circle Packing Theorem (TCPT)]
	Let $(P, w)$ be a weighted abstract triangulated polyhedron where $w(ij)=1$ for all $ij\in E(P)$. Then there exists a geodesic, strictly convex circle polyhedron $(P, {\bf D})$ realizing $(P, w)$. 
\end{Theorem}
From this starting point we now show that the flow used by Connelly and Gortler to obtain their proof of TCPT may be extended past tangency to bring disks into shallow overlaps. Given an abstract shallow-weighted triangulated polyhedron $(P, w)$ our starting point is the tangency packing with $P$ as its contact graph guaranteed by TCPT. From there we apply a series of flow operations to correct, edge-by-edge, the inversive distances from tangency to whatever the desired inversive distance is in $(P, w)$. If $(P, w)$ is strictly shallow, then we can correct each edge by a single flow. When $(P, w)$ is not strictly shallow, but has some desired inversive distances of $0$, we require a more nuanced limiting argument.

\subsection{Obtaining strictly shallow circle polyhedra using flows}

We first prove a version of our main theorem restricted to strictly shallow polyhedra. In the next section we extend this result to handle non-strictly shallow polyhedra. The key difference is that strictly shallow polyhedra that are strictly convex maintain strict convexity under certain motions due to theorem~\ref{thm:convexpreservation} and our main argument in this section requires strict convexity in order to apply the infinitesimal rigidity results from the last section.  

\subsubsection{The free-edge and unmarked rigidity matrices}

Let $(P, {\bf D})$ be a strictly shallow, strictly convex circle polyhedron. Let $J_q$ denote the square rigidity matrix of $(P, {\bf D})$ with respect to some face $q\in F(P)$. By lemma~\ref{lem:squarerigiditymatrixrank} the rank of $J_q$ is $4n-6$. Now, consider removing one of the edge constraint rows imposed by $P$ for some edge $e=ij\in E(P)$. The result is to remove the row of $J_q$ corresponding to $e$. Denote $P$ with the edge $e$ removed by $P^{-e}$ and $J_q$ with the $e$ row removed by $J^{-e}$. Since $J_q$ has full rank, so does $J^{-e}$. Call $J^{-e}$ the \term{free-edge rigidity matrix} of $(P, {\bf D})$ with respect to $q$ and $e$. We have that:

\begin{Lemma}
	The free-edge rigidity matrix $J^{-e}$ of a convex circle polyhedron $(P, {\bf D})$ with respect to a face $q$ and edge $e$ is non-singular.
\end{Lemma}

Finally, as in \cite{CG19}, we note that the matrix $J_q$ has a useful block form. Let $i$, $j$, and $k$ denote the vertices of the pinned face $q$. Each of the edge rows of $J_q$ corresponding to the edges $ij$, $jk$, and $ki$ have non-zero entries only in the columns of $J_q$ that correspond to the coordinates of the disks $D_i$, $D_j$, $D_k$ that were not already removed from $J$ to obtain $J_q$. This is also the case for the rows corresponding to the vertices $i$, $j$, and $k$ designed to keep $D_i$, $D_j$, and $D_k$ in the de Sitter sphere. For $D_i$, $D_j$, and $D_k$, six of their coordinate columns were removed from the rigidity matrix $J$ to obtain $J_q$. Simultaneously removing the remaining six coordinate columns along with the edge rows for $ij$, $jk$, and $ki$ and the vertex rows for $i$, $j$, and $k$ from $J_q$ leaves us with a square matrix of size $4n-12$. Its edge rows correspond to the edges of $P$ that have at least one endpoint not in $\{i, j, k\}$. Its vertex rows and columns correspond to the vertices of $P$ that are not $i$, $j$, or $k$. Call this the \term{unmarked rigidity matrix} and denote it by $J_u$. From the block form and previous lemma we have that: 

\begin{Lemma}\label{lem:unmarkednonsingular}
	The unmarked rigidity matrix $J_u$ of a strictly convex circle polyhedron $(P, {\bf D})$ with respect to a face $q$ and edge $ij$ is non-singular.
\end{Lemma}

\subsubsection{The shallow packing manifold}

The configuration space $(P, {\bf D})$ with three disks $D_i$, $D_j$, and $D_k$ fixed is $\mathbb{R}^{4n-12}$. 

Now, select an edge $e\in E(P)$ that is not $ij$, $jk$, or $ki$. Denote by $S_P^{-e}$ the subset of the configuration space that represent all configurations $(P, {\bf D}')$ that are strictly shallow strictly convex geodesic circle polyhedra and for every edge $e'\in E(P)$ that is not $e$, the inversive distance between the disks corresponding to the endpoints of $e'$ in $(P, {\bf D}')$ is the same as the inversive distance between the endpoints of $e'$ in $(P, {\bf D})$. In other words, $(P, {\bf D})$ and $(P, {\bf D}')$ agree in inversive distances on all but one edge. We enumerate several conditions on $(P, {\bf D}')$ that follow this construction. Because $(P, {\bf D}')$ is strictly shallow and geodesic, by lemma~\ref{lem:localtoglobal} non-adjacent disks are disjoint.

We remark that the conditions for a point to be in $S_P^{-e}$ are open conditions on the $n-3$ free disks. We now analyze the properties of $S_P^{-e}$. 
\begin{Lemma}
	$S_P^{-e}$ is a 1-dimensional smooth manifold.
\end{Lemma}
\begin{proof}
The argument here is essentially Connelly and Gortler's with modified details. Let $(P, {\bf D})$ be a circle polyhedron corresponding to a point $p\in S_P^{-e}$ and consider a sufficiently small neighborhood $U$ of $p$ in $S_P^{-e}$. Every configuration $(P, {\bf D}')$ corresponding to a point of $U$ has a radius for each of the free disks in ${\bf D}'$ in $(0, 2\pi)$ and all non-edge inversive distances are greater than $1$. Locally we need only consider the $3n-10$ inversive distances corresponding to the edges of $P$ that are not the free-edge $e$ and are not incident the pinned face $ijk$ and the constraints keeping each of the $n-3$ free disks of ${\bf D}'$ constrained to the de Sitter sphere. The partial derivatives of these constraints are the $4n-13$ rows of $J_u$ corresponding to the edges that are not $e$, which by lemma~\ref{lem:unmarkednonsingular} is non-singular. The result then follows by the implicit function theorem.  
\end{proof}

\subsubsection{The flow}

Let $(P, w)$ be an abstract weighted triangulated polyhedron. Let $(P, {\bf D})$ be a circle polyhedron with the same abstract polyhedron $P$. We call $(P, {\bf D})$ \term{$w$-bounded} if $d(D_i, D_j) \geq w(ij)$ for every edge $ij\in E(P)$. 	

We now use the Connelly-Gortler flow to adjust inversive distances. Unlike their use, we do not need to make combinatorial changes to our polyhedron, since our starting point is a tangency packing with the correct combinatorics. We simply adjust the inversive distances along one free edge at a time, keeping the inversive distance across all other edges fixed.  

Assume we have a desired set of weights $w$ we want to see realized as inversive distances for each edge in $P$. Throughout our use of flows we maintain the invariant that our polyhedron $(P, {\bf D})$ is $w$-bounded. In the remainder let $d(e)$ denote the inversive distance across an edge $e$ in $(P, {\bf D})$, $d(e(t))$ denote the inversive distance across the edge in $(P, {\bf D}(t))$, and $w(e)$ denote the desired inversive distance.

We begin with a $w$-bounded polyhedron $(P, {\bf D})$. Select an edge $e$ for which $d(e) > w(e)$. We apply a motion called a flow (defined below) to $(P, {\bf D})$ that strictly decreases the inversive distance along $e$ while maintaining the inversive distance across every other edge. In order to mod out the Möbius transformations from this motion, we pin three disks corresponding to a face of $ijk \in P$. This choice necessarily fixes the inversive distances across the edges $ij$, $jk$, and $ki$. Therefore we choose $ijk$ so that $e$ is not one of its edges.

With three disks pinned, the configuration of ${\bf D}$ is described by a point in $\mathbb{R}^{4n-12}$. Let $J_u$ be the unmarked rigidity matrix of $(P, {\bf D})$ with respect to $ijk$. Since $J_u$ is non-singular, its inverse exists. Let ${\bf D}':=J_u^{-1}{\bf v}_e$ where ${\bf v}_e$ is the vector whose entries are all 0 except that the entry corresponding to the edge $e$ is -1. ${\bf D}'$ defines a velocity field on the $n-3$ non-fixed disks of $(P, {\bf D})$ that keeps all edge inversive distances fixed save the free edge $e$, which has an inversive distance that decreases at a constant rate. Since $J_u$ is non-singular at $(P, {\bf D})$ it is non-singular over a Zariski open subset $U$ of $\mathbb{R}^{4n-12}$ and thus defines a smooth velocity field over $U$ including $(P, {\bf D})$ that we use to set up a system of ordinary differential equations (ODEs). 

Starting at $(P, {\bf D})$ and integrating forward in time obtains a maximal trajectory $(P, {\bf D}(t))$ for some time interval $[0, T)$ where $T\in\mathbb{R}\cup\{\infty\}$, which we call the \term{flow trajectory}. By the standard theory of ODEs this trajectory will either leave any compact set in $U$ or continue for infinite time. 

Since the inversive distance $d(e(t))$ decreases at a constant rate while all other edge inversive distances are fixed, at some finite time $t'\in [0, T)$ we will either achieve an inversive distance of $w(e)$ along the edge $e$ in $(P, {\bf D}(t'))$ or $T$ is finite and $(P, {\bf D}(T)) \not\in S_P^{-e}$. In the former case, we flow along the motion until from $t = 0$ to $t'$ to obtain a new polyhedron $(P, {\bf D}(t'))$ at which point $d(e(t')) = w(e)$. In the latter case $(P, {\bf D}(T))$ is either not strictly shallow, not strictly convex, or some disk has degenerated to a point. We now add conditions necessary to ensure that this latter case does not occur. 

\begin{Lemma}\label{lem:flowdefined}
Let $(P, w)$ be a strictly shallow weighted abstract triangulated polyhedron, not a tetrahedron, and let $(P, {\bf D})$ be a strictly shallow strictly convex geodesic circle polyhedron that is $w$-bounded. Let $e$ be an edge such that $d(e) > w(e)$ and $ijk$ be a pinned face that is not bounded by $e$.

Suppose further that if $e_1$, $e_2$, $e_3$ form a closed loop of edges such that $\sum_{k=1}^3 \arccos w(e_k) \geq \pi$ then $e_1$, $e_2$, and $e_3$ bound a face of $P$.

Then at a finite $t'$ the flow reaches a point of the trajectory such that $d(e(t')) = w(e)$ and for all $0\leq t'' \leq t'$ $(P, {\bf D}(t''))$ is in $S_P^{-e}$.  
\end{Lemma}

In other words, if we have the first condition from the Koebe-Andre'ev-Thurston theorem to our desired weight function $w$, then our flow trajectory can be used to correct any edge not yet matching our desired weight. (The second condition from the Koebe-Andre'ev-Thurston theorem on 4-cycles of edges does not apply in this case because our polyhedra are strictly shallow. For a four-cycle of edges to have overlaps summing to $2\pi$ when all overlap angles are between $0$ and $\pi/2$ requires that all overlap angles be $\pi/2$ and thus not strictly shallow.)
\begin{proof}
Let $[0, T]$ be an interval of the flow trajectory in which $d(e(t)) \geq w(e)$ for all $t\in [0, T]$. We start strictly shallow, and we maintain all inversive distances except along $e$, which itself remains strictly shallow (since it is decreasing and $d(e(t)) > w(e) > 0$). Therefore $(P, {\bf D}(t))$ is strictly shallow for all $[0, T]$. Then by theorems~\ref{thm:geodesicpreservation} \& \ref{thm:convexpreservation}, the trajectory remains geodesic and strictly convex. 

Then, by our extra condition that closed loops of three edges having angle sums of $\geq \pi$ must bound a face, and the fact that the only inversive distance that is changing is along $e$, condition 1 of theorem~\ref{thm:katconditionsnocollapse} is maintained for $t\in[0, T]$, and therefore no disk vanishes throughout the motion. (Condition 2 is maintained vacuously because our polyhedra and desired weights $w$ are strictly shallow.) 

Then $(P, {\bf D}(t))$ is in $S_P^{e-}$ for every finite interval maintaining $d(e(t)) \geq w(e)$. But the $d(e(t))$ is decreasing at a constant rate as $t$ increases. Therefore there must be a finite time $t'$ such that $d(e(t')) = w(e)$.
\end{proof}

We now prove the Koebe-Andre'ev-Thurston theorem for strictly shallow circle polyhedra. In the next section, we extend this to the full Koebe-Andre'ev-Thurston theorem. 

\begin{Theorem}[Strictly shallow Koebe-Andre'ev-Thurston]\label{thm:strictlyshallowkat}
	Let $(P, w)$ be a strictly shallow abstract triangulated polyhedron, not a tetrahedron, and let the following condition hold. If $e_1$, $e_2$, $e_3$ form a closed loop of edges with $\sum_{i=1}^3 \arccos w(e_i) \geq \pi$, then $e_1$, $e_2$, and $e_3$ bound a face of $P$. 
	Then there exists a geodesic circle polyhedron $(P, \mathbf{D})$ realizing $(P, w)$. 
\end{Theorem}
\begin{proof}
	Start with the geodesic, strictly convex tangency circle polyhedron $(P, {\bf D})$ with the combinatorics $P$ guaranteed by the tangency circle packing theorem. Since all inversive distances in $(P, {\bf D})$ are 1,  $(P, {\bf D})$ is $w$-bounded.
	 
	We now produce a finite sequence of polyhedra $(P, {\bf D}_1), (P, {\bf D}_2, \dots, (P, {\bf D}_k)$ such that $(P, {\bf D}_1) = (P, {\bf D})$ and $(P, {\bf D}_k)$ realizes $(P, w)$ and each of which agrees with $(P, w)$ on more edge than the previous polyhedron in the sequence. Let $e$ be an edge of $(P, {\bf D}_i)$ such $d(e) > w(e)$. By our hypothesis (which is maintained inductively), by lemma~\ref{lem:flowdefined} we may apply the flow $(P, {\bf D}_i(t))$ to continuously decrease $d(e(t))$ until it equals $w(e)$, thus obtaining the next polyhedron $(P, {\bf D}_{i+1})$ which agrees with $w$ on one more edge (namely $e$). In the worst case, we must apply this operation $|E(P)|$ times to correct one edge at a time until they all match $w$.  
\end{proof}

\paragraph{Remark.} As pointed out in a personal correspondence by Steven Gortler it is likely that the ingredients above may be used to flow all disks from tangency to their desired inversive distances nearly simultaneously. First select a single triangle of $P$ and flow to obtain the desired inversive distances along its edges. Then fix this triangle and flow the remaining disks. We leave the details of this approach to future work. 

\subsection{Obtaining non-strictly shallow polyhedra using flows}

{\em Proof of Koebe-Andre'ev-Thurston.} We now use a limiting argument to obtain the proof of the full Koebe-Andre'ev-Thurston theorem. Let $(P, w)$ be an abstract weighted triangulated polyhedron that is shallow, but not strictly shallow and satisfies the extra conditions of KAT. Then there is some 4-cycle of edges $ij$, $jl$, $lk$, $ki$ bounding faces $ijk$ and $kjl$ such that $w(ij) = w(jl) = w(lk) = w(ki) = 0$. Let $w_1:E(P)\rightarrow\mathbb{R}^+$ be the weight function defined by $w_1(ij) = w_1(ki) = 0$, $w_1(jk)=w(jk)$, and $w_1(e)=1$ for all other edges. Then $(P, w_1)$ is a strictly shallow weight function and by theorem~\ref{thm:strictlyshallowkat} we obtain a strictly shallow strictly convex geodesic realization $(P, {\bf D}_1)$. By construction, this realization now has the correct desired inversive distances along the edges of the face $ijk\in F(P)$. 

Now, for any fixed $\epsilon > 0$ let $w_\epsilon$ denote the function such that $w_\epsilon(ij) = w_\epsilon(ki) = 0$, $w_\epsilon(jk) = 1$ and for all other edges $e$, $w_\epsilon(e) = \max\{\epsilon, w(e)\}$. Consider the sequence of weighted abstract polyhedra $(P, w_1), (P, w_{1/2}), (P, w_{1/4}), \dots$. Starting with $(P, \mathbf{D}_1)$ we produce a sequence of circle polyhedra $(P, {\bf D}_1), (P, {\bf D}_{1/2}), (P, {\bf D}_{1/4}), \dots$ as follows. Each $(P, {\bf D}_{1/2^{i+1}})$ is obtained from the previous $(P, {\bf D}_{1/2^i})$ by pinning the face $ijk$ and flowing across each of the remaining edges that are not yet at their correct inversive distance. Thus, between each consecutive pair of polyhedra in the sequence, we have a continuous motion of disks which is geodesic, strictly shallow, and strictly convex. Throughout the motion each inversive distance is (non-strictly) decreasing. Furthermore, for any edge $e$ such that $w(e) > 0$, after finitely many polyhedra in the sequence the edge $e$ achieves $w(e)$ as its inversive distance which is then held constant throughout the rest of the motion. For the edges where $w(e)=0$ (that are not $ij$ or $ki$) the inversive distances approach $0$ as the motion continues.

By the Bolzano-Weierstrass theorem, there is a convergent subsequence of this motion that converges to a set of disks $D_v'$ for each $v\in V(P)$. If we require that $(P, w)$ satisfies conditions 1 and 2 of the Koebe-Andre'ev-Thurston theorem, then our motion satisfies conditions 1 and 2 of theorem~\ref{thm:katconditionsnocollapse} and thus no disk vanishes. Therefore every disk $D_v'$ is a real disk and for every edge $uv\in E(P)$ we must have that $d(D_u', D_v') = w(uv)$.  

This completes the proof of theorem~\ref{thm:kat}.
\qed
\section{Conclusion}\label{sec:conclusion}

This paper presents a new proof of the general Koebe-Andre'ev-Thurston theorem which extends the ideas of Connelly and Gortler's proof in the tangency case to handle circle packings with overlaps. Still open are the problems of existence (unsolved) and uniqueness (solved only for non-unitary convex circle polyhedra) of circle packings with mixes of overlapping and non-overlapping circles. Unlike the overlap packings of the Koebe-Andre'ev-Thurston theorem, such polyhedra need not be convex nor geodesic. That said, however, the flow trajectory can, in fact, be used to produce packings with a mix of overlaps and non-overlaps, both for convex geodesic circle polyhedra and circle polyhedra that are neither convex nor geodesic. An example is shown in figure~\ref{fig:algorithmrun2}. This is an intriguing experimental observation and we hope that the present work may serve as a further launching off point in the effort to characterize the existence and uniqueness of circle packings on the sphere.

Also, in this paper we have made extensive use of two related notions: geodesy and convexity. The Ma-Schlenker counter-examples to the global rigidity of general inversive distance circle packings give rise to a family of examples (explored in \cite{Bowers2017}) that are non-convex and have realizations that are geodesic, as well as realizations that are non-geodesic. Contrary to such examples, our investigation in the present paper has shown that in some cases geodesic circle polyhedra and convex circle polyhedra are intimately related. In some sense this seems to connect to the all of a convex Euclidean polyhedra are visible to any point on its interior, whereas some points on the interior of a non-convex Euclidean polyhedra may see all of the polyhedra, but there always exist some points on the interior for which some of the polyhedron is obstructed. We leave the reader with one further intriguing question: is it the case that a circle polyhedron is convex if and only if all Möbius transformations of the polyhedron are geodesic? This would align with both the Ma-Schlenker examples and much of the development in this paper, but remains an open question.

\begin{figure}[hbt!]
  \centering
  \begin{minipage}[b]{0.3\textwidth}
  	\includegraphics[width=\textwidth]{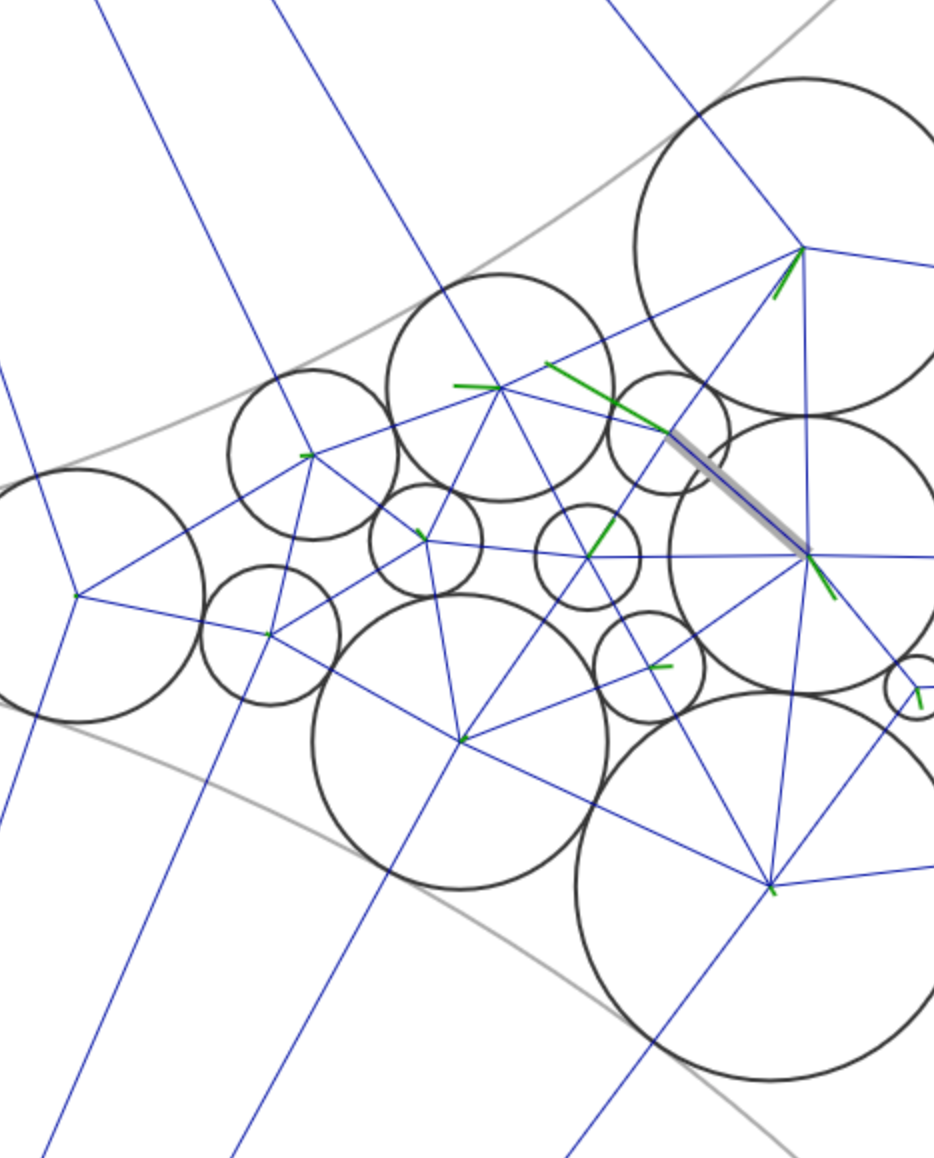}
	\subcaption{}
  \end{minipage}
  \begin{minipage}[b]{0.3\textwidth}
	  \includegraphics[width=\textwidth]{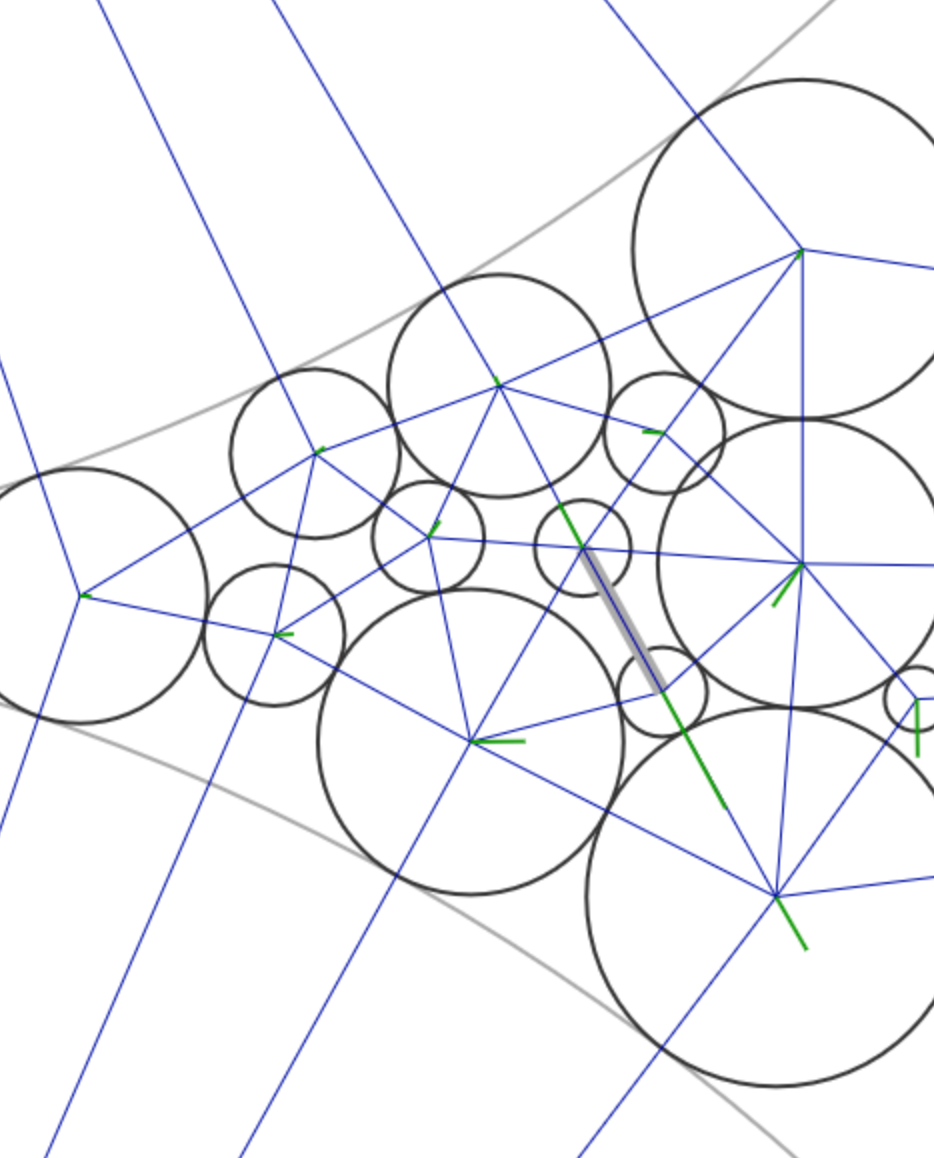}
	  \subcaption{}
  \end{minipage} \\
  \begin{minipage}[b]{0.3\textwidth}
  	\includegraphics[width=\textwidth]{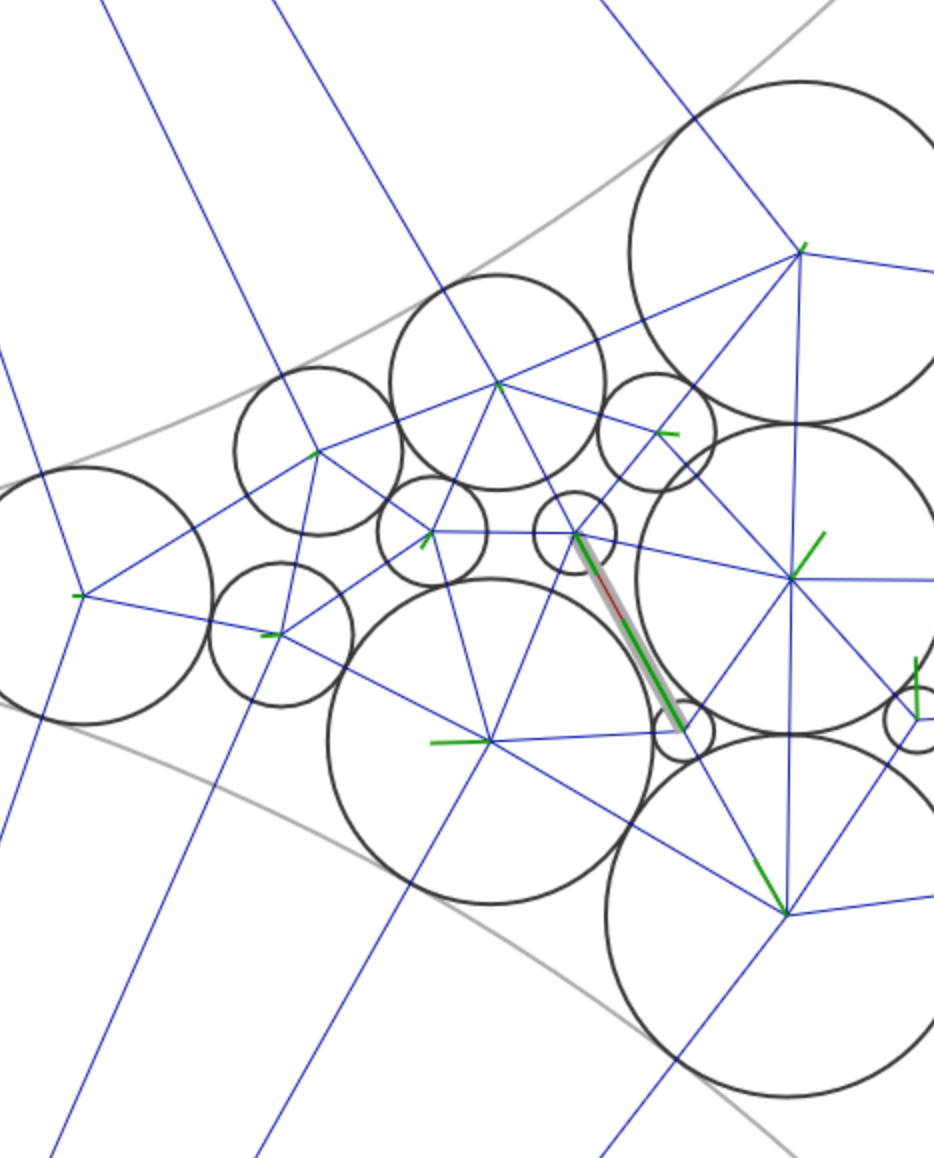}
	\subcaption{}
  \end{minipage}
  \begin{minipage}[b]{0.3\textwidth}
	  \includegraphics[width=\textwidth]{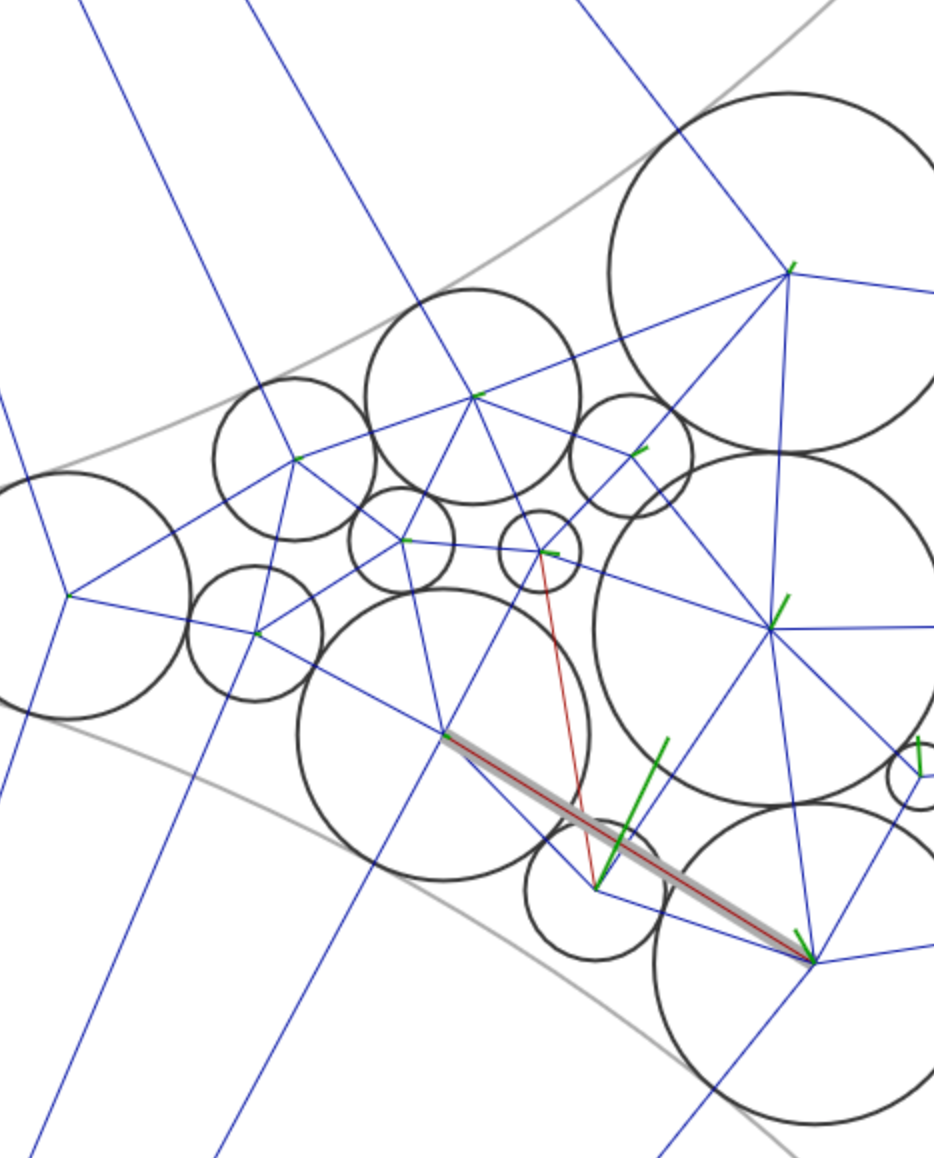}
	  \subcaption{}
  \end{minipage} \\
  
  \caption{Our proof guarantees that we can move along the flow trajectory to obtain any desired shallow packing, but do not rule out using the flow to obtain packings with non-shallow inversive distances. This figure shows the application of a flow to obtain a circle polyhedron with inversive distances that are shallow and inversive distances that are larger than $1$. (a) is a configuration obtained numerically by the the flow. The gray edge represents our free edge and the green vectors from the center of each circle represent the vector of the motion that is maintaining all other inversive distances. In (b), (c), and (d) we continue to apply the flow on the same edge. In (b), the polyhedron is still convex and geodesic. In (c), the polyhedron has ceased being convex (we would need to flip the gray edge at an intermediate point between (b) and (c) to maintain convexity. In (d) the polyhedron is both non-convex and non-geodesic as the center of one circle has moved through the edge of one of its link triangles. }
  \label{fig:algorithmrun2}
\end{figure}

\paragraph{Acknowledgements.} We especially thank Philip Bowers for a great many helpful conversations, challenges, and encouragement during the development of the present work. We also very much thank Bob Connelly for the invitation to Cornell in the fall of 2019 that led the author to begin developing code for experimenting with flip-and-flow and lead eventually to this paper.
We thank Ken Stephenson for pointing out a Bolzano-Weierstrass based proof of the Ring Lemma for tangency packings which informed our analysis of disk motions. We also thank Steven Gortler, Tony Nixon, and Bernd Schulze for helpful conversations and guidance. Finally, we thank both Brown University's Institute for Computational and Experimental Research Mathematics (ICERM) and the American Institute of Mathematics (AIM) and the organizers for hosting meetings on circle packing and rigidity theory that indirectly led to the development of this paper. 

\bibliography{OurBib}{}
\bibliographystyle{plain}

\appendix

\end{document}